%% file: dwe_arxiv02052025.tex
\documentclass[a4paper,reqno]{amsart}

\input{commands.tex}      

\usepackage{mathtools}
\mathtoolsset{showonlyrefs}

\numberwithin{equation}{section}
\numberwithin{figure}{section}

\begin{document}
\title[Resolvent estimates for the damped wave equation]{Resolvent estimates for the one-dimensional damped wave equation with unbounded damping}

\author{Antonio Arnal}

\address{Mathematical Sciences Research Centre, Queen's University Belfast, University Road, Belfast BT7 1NN, UK}

\email{aarnalperez01@qub.ac.uk}

\thanks{The author should like to express his gratitude to his supervisor, P. Siegl, Graz University of Technology and Queen's University Belfast, for very valuable comments and suggestions in the course of the work presented in this paper, and to B. Gerhat, Czech Technical University in Prague, for introducing the framework described in \cite{gerhat2022schur} to him.}



\date{\today}

\begin{abstract}
	We study the generator $\opG$ of the one-dimensional damped wave equation with unbounded damping. We show that the norm of the corresponding resolvent operator, $\| (\opG - \la)^{-1} \|$, is approximately constant as $|\la| \to +\infty$ on vertical strips of bounded width contained in the closure of the left-hand side complex semi-plane, $\overline\C_{-} := \{\la \in \C: \Re\la \le 0\}$. Our proof rests on a precise asymptotic analysis of the norm of the inverse of $\opT(\la)$, the quadratic operator associated with $\opG$.
\end{abstract}

\maketitle

\section{Introduction}
\label{sec:intro}
There is a well-developed theory for self-adjoint linear operators based on a number of key tools, notably the spectral theorem. This fundamental result underlies the fact that the spectrum of a self-adjoint operator contains a significant amount of information which is of great value to understand its action. It is equally well-known that there is no equivalent result for non-self-adjoint (NSA) linear operators. This deficiency is related to the spectral instability under small perturbations often exhibited by NSA operators. Such behaviour has prompted the development of new tools and techniques to study them, the pseudospectrum being one of the most widely used (see e.g. \cite{Davies-2000-43, Davies-2002-34, Trefethen-1997-39}). We recall that, if $\opH$ is a closed linear operator acting in a Hilbert space $\cH$ and we take $\eps > 0$, the $\eps$-pseudospectrum of $\opH$ is defined as
\begin{equation*}
	\sigma_{\eps}(\opH) := \sigma(\opH) \cup \{\la \in \C : \| (\opH - \la)^{-1} \| > \eps^{-1}\}.
\end{equation*}
It is immediate from its definition that $\sigma_{\eps}(\opH)$ is a family of nested open sets which increase as $\eps \to +\infty$ and approach $\sigma(\opH)$ as $\eps \to 0$. These sets can also be characterised as follows
\begin{equation*}
	\sigma_{\eps}(\opH) = \{\la \in \C : \la \in \sigma(\opH + \opA) \text{ for some } \| \opA \| < \eps\}
\end{equation*}
(see \cite[Thm.~13.2]{Helffer-2013-book}), which makes apparent why one can expect the pseudospectrum to be more robust under linear perturbations than the spectrum.

In this context, it also becomes clear that the spectral analysis of any NSA operator must include at least some quantitative understanding of the behaviour of the norm of the resolvent operator $\| (\opH - \la)^{-1} \|$ for $\la$ in the resolvent set $\rho(\opH)$. Using general operator-theoretic arguments, it is possible to show that, if $\opH$ is a closed operator whose numerical range, $\Num(\opH)$, satisfies that each connected component of $\C \setminus \overline{\Num(\opH)}$ has non-empty intersection with $\rho(\opH)$, then
\begin{equation*}
	\| (\opH - \la)^{-1} \| \le \frac1{\dist(\la, \overline{\Num(\opH)})}, \quad \la \in \rho(\opH),
\end{equation*}
(see \cite[Thm.~III.2.3]{EE}). This estimate has the weakness that it provides very limited information about the behaviour of $\| (\opH - \la)^{-1} \|$ when $\la$ lies near the boundary of the numerical range and none at all when it is inside. An aim of recent research in this area has been to shed light on such behaviour, using both semi-classical (e.g. \cite{Davies-1999-200, Dencker-2004-57, Sjoestrand-2009, BordeauxMontrieux-2013}) and non-semi-classical (e.g. \cite{Krejcirik-2019-276, KREJCIRIK2022109440, Arifoski-2020-52, duc2022pseudomodes, ArSi-resolvent-2022}) methods. One approach, pioneered in \cite{Davies-1999-200} and subsequently developed non-semi-classically in \cite{Krejcirik-2019-276, Arifoski-2020-52, KREJCIRIK2022109440, duc2022pseudomodes}, relies on the construction of pseudomodes (or approximate eigenfunctions) for the operator at hand (Schr\"odinger, damped wave equation, Dirac, biharmonic) inside the numerical range thereby finding lower bounds on $\| (\opH - \la)^{-1} \|$. For Schr\"odinger operators with complex potentials, lower \textit{and} upper bounds have recently been found in \cite{ArSi-resolvent-2022} using different (non-semi-classical) methods.

The aim of the work presented in this paper is to apply the new techniques developed in \cite{ArSi-resolvent-2022} to the study of the resolvent of the NSA generator $\opG$ for the one-dimensional damped wave equation (DWE) described by
\begin{equation}
	\label{eq:dwe.2ndorder}
	\Dtime u(t,x) + 2 a(x) \Ntime u(t,x) = (\Dtp - q(x)) u(t,x), \quad t > 0, \quad x \in \R,
\end{equation}
with non-negative damping $a$ unbounded at infinity and non-negative potential $q$ which may also be unbounded. There is a great deal of research literature covering the case where $a$ is a bounded function on a (possibly unbounded) domain $\Omega$ in $\Rd$, $d \ge 1$, reflecting applications where the solution to the corresponding initial value problem decays exponentially with time. On the other hand, recent research (see \cite{sobajima2018diffusion, Ikehata-2018, Freitas-2018-264, Arifoski-2020-52}) has focused on the study of the equation when $a$ is unbounded at infinity and on the impact of this feature on the spectral structure of the generator and/or the large-time behaviour of solutions. In \cite{Ikehata-2018}, the existence and uniqueness of a weak solution to the initial value problem for \eqref{eq:dwe.2ndorder}, with some mild assumptions on the initial data, were proven for continuous damping bounded below by a positive constant in $\Rd$, $d \ge 3$. Furthermore, it was shown that both the solution and its energy decay polynomially with time. A similar result was presented in \cite{sobajima2018diffusion} for dampings of type $a(x) = a_0 |x|^{\alpha}$, with $a_0, \alpha > 0$, on exterior domains in $\Rd$, $d \ge 2$, although assuming more restrictive conditions on the initial data. In \cite{krejcirik2022spectrum}, the authors carry out an spectral analysis of the wave equation with distributional (Dirac $\delta$) damping on a non-compact star graph that highlights the \textit{wild spectral} behaviour associated with its generator's non-self-adjointness; upper and lower bounds for the resolvent norm are also found (see \cite[Thm.~2.3]{krejcirik2022spectrum}). The perspective and methods used in \cite{Freitas-2018-264} are closer to those applied in this paper, exploring as they do the impact of the behaviour at infinity of $a$ on the emergence of the essential spectrum of $\opG$ and the stability of solutions. A similar spectral and stability analysis was carried out in \cite{Freitas-2020-148} for dampings of type $a(x) = \alpha/x$, $\alpha > 0$, on $\Omega = (0,1) \subset \R$. The pseudospectrum of $\opG$ for a wide class of unbounded dampings $a(x), \; x \in \R,$ was studied in \cite{Arifoski-2020-52} using a pseudomode construction $\{\psi_{\la} \in \Dom(\opG) : \la \in \Gamma \subset \C\}$ of WKB type and estimates were obtained (as $\la \to \infty$, $\la \in \Gamma$) for the decay rate of
\begin{equation*}
	\frac{\| (\opG - \la) \psi_{\la} \|_{\cH}}{\| \psi_{\la} \|_{\cH}}
\end{equation*}
(with $\cH$ denoting the underlying Hilbert space where the operator $\opG$ acts) to yield a lower bound on $\| (\opG - \la)^{-1} \|$.

The main finding in this paper is formulated in Theorem~\ref{thm:resolvent.G} and concerns the asymptotic behaviour of $\| (\opG - \la)^{-1} \|$ in $\overline \C_{-}$. For dampings $a$ obeying Assumption~\ref{asm:a.q.dwe}, which encompasses smooth unbounded non-negative real functions with controlled derivatives (e.g. $a(x) = x^{2n}, \; n \in \N$, see Asm.~\ref{asm:a.q.dwe}~\ref{itm:a.symbolclass}), we prove that $\| (\opG - \la)^{-1} \|$ is approximately constant in any bounded-width vertical strip in $\overline \C_{-}$ as $|\Im \la| \to +\infty$. Our result supports smooth non-negative potentials $q$ with controlled derivatives (Asm.~\ref{asm:a.q.dwe}~\ref{itm:q.symbolclass}) as long as they are "no stronger" than $a$ at infinity (Asm.~\ref{asm:a.q.dwe}~\ref{itm:a.gt.q}). It encompasses cases where $\| (\opG - \la)^{-1} \|$ can be shown to diverge along any ray in the second (or third) quadrant of the complex plane using the methods in \cite{Arifoski-2020-52} (see Remark~\ref{rmk:rays.growth}). The key element in our proof of Theorem~\ref{thm:resolvent.G} is the derivation in Theorem~\ref{thm:resolvent.Tla} of an asymptotic estimate for the norm of the inverse of the quadratic operator
\begin{equation*}
	\opT(\la) = \Dt + q(x) + 2 \la a(x) + \la^2, \quad \la \in \C \setminus (-\infty,0].
\end{equation*}
Although we shall defer a more rigorous definition of this operator, along with an explanation of how it relates to $\opG$, until Sub-section~\ref{ssec:dwe.prelim}, we observe here that its structure is that of a $\la$-dependent Schr\"odinger operator with the complex potential $q(x) + 2 \la a(x) + \la^2$. Whereas the fact that $q$ is "no stronger" than $a$ discourages us from (for example) attempting to recast the problem as a (relatively bounded) perturbation of a self-adjoint operator, it does on the other hand broadly fit into the framework used to prove \cite[Thm.~4.2]{ArSi-resolvent-2022}, where the asymptotic behaviour of $\| (\opH - \la)^{-1} \|$ along the real axis for a one-dimensional Schr\"odinger operator with a complex potential $\opH$ was determined. In order to adapt to $\opT(\la)$ the strategy introduced in that paper, we begin by transforming the problem to Fourier space (see \eqref{eq:Tlahat.def}). The resulting pseudo-differential operator $\widehat \opT(\la) = \widehat{q} + 2 \la \widehat{a} + \xi^2 +\la^2$ has the potential term $\xi^2 + \la^2$ (for $\la := -c + i b$) with turning points $\pm \xi_b$ (where $\xi_b := b$). We subsequently carry out a separate analysis of $\| \widehat \opT(\la) u \|$ depending on whether or not $\supp u$ is contained in certain neighbourhoods of $\pm \xi_b$ designed so that $\xi^2$ is approximately constant inside. More specifically, the proof of Theorem~\ref{thm:resolvent.Tla} consists of the following steps (with $\la = -c + ib$, where $c \in K \subset \overline \Rplus$, $K$ bounded, and $b \in \R \setminus \{0\}$): 
\begin{enumerate}[\upshape (1),wide]
	\item In Proposition~\ref{prop:away.Tla}, with $\Omega'_{b,\pm}$ representing the neighbourhoods of $\pm\xi_b$ defined in \eqref{eq:deltab.def}, we use direct $\Lt$-norm estimates to find that as $b \to +\infty$
	\begin{equation*}
		b^2 \ls_\delta \inf \left\{ \frac{\| \widehat \opT(\la)  u \|}{\|u\|} : \; 0 \neq u \in \Dom(\widehat \opT(\la)), \; \supp u \cap (\Omega'_{b,+} \cup \Omega'_{b,-}) = \emptyset \right\}.
	\end{equation*}
	\item \label{itm:step2.intro} In Proposition~\ref{prop:local.Tla}, inside neighbourhoods $\Omega_{b,\pm}$ of $\pm\xi_b$ defined in \eqref{eq:Omega.def} and appropriately shifted, we Taylor-approximate $\widehat \opT(\la)$ with the (Fourier-space) pseudo-differential version of the generalised Airy operator, $\opA = \Nt + a(x)$, shifted by $c$ to yield as $b \to +\infty$
	\begin{equation*}
		\begin{aligned}
			&\| (\opA - c)^{-1} \|^{-1} 2 b (1 - \BigO_K(b^{-1}))\\
			&\qquad \qquad \le \inf \left\{ \frac{\| \widehat \opT(\la) u \|}{\|u\|} : \; 0 \neq u \in \Dom(\widehat \opT(\la)), \, \supp u \subset \Omega_{b,\pm} \right\}.
		\end{aligned}	
	\end{equation*}
	The norm resolvent convergence of (a localised realisation of) $\widehat \opT(\la)$ to the pseudo-differential version of $\opA - c$ follows from the second resolvent identity and several graph norm estimates obtained by standard arguments.
	\item In Proposition~\ref{prop:lbound.Tla}, we show that the estimate for $\| \widehat\opT(\la)^{-1} \|$ obtained in Step~\ref{itm:step2.intro} cannot be improved by finding functions $u_b \in \Dom(\widehat\opT(\la))$ such that as $b \to + \infty$
	\begin{equation*}
		\| \widehat\opT(\la) u_b \| = \| (\opA - c)^{-1} \|^{-1} 2 b (1 + \BigO_K(b^{-1})) \| u_b \|.
	\end{equation*}
	The proof relies on exploiting the localisation technique applied in Proposition~\ref{prop:local.Tla} and the fact that the operators involved have compact resolvent. Thus the norms of those resolvents can be obtained from the appropriate singular values and the corresponding eigenfunctions are used to determine the family $u_b$ with the aid of certain cut-off functions.
	\item In our final step, we combine the results from the previous ones with the aid of commutator estimates and a suitably constructed partition of unity.
\end{enumerate}

The remainder of our paper is structured as follows. Section~\ref{sec:prelim} describes our notation and recalls some fundamental facts for the DWE and for the various tools (e.g. generalised Airy operators) used throughout. Section~\ref{sec:assumptions} formulates our assumptions, states our main result for the generator $\opG$ and draws some consequences for the long-time behaviour of the associated $C_0$-semigroup that solves the corresponding Cauchy problem. Section~\ref{sec:quadop} is devoted to investigating a number of important properties for the quadratic operator $\opT(\la)$ associated with $\opG$, including a crucial result regarding the asymptotic behaviour of the norm of its inverse in bounded-width vertical strips inside $\overline\C_{-}$ which is subsequently extended to general curves adjacent to the imaginary axis (see Sub-section~\ref{ssec:resnorm.adj.dwe}). The proof of our main theorem can be found in Section~\ref{sec:mainthmG}. Section~\ref{sec:examples} illustrates our results with a detailed analysis of an example (with $a(x) = x^2$ and $q(x) = \kappa x^2$, $\kappa > 0$) where $\sigma(\opG)$ is calculated and the stability of the $C_0$-semigroup discussed.

\section{Notation and preliminaries}
\label{sec:prelim}
We write $\N_0:=\N \cup \{0\}$, $\R_{+} := (0, +\infty)$, $\R_{-} := (-\infty, 0)$, $\C_{+} := \{\la \in \C: \Re\la > 0\}$ and $\C_{-} := \{\la \in \C: \Re\la < 0\}$. The characteristic function of a set $E$ is denoted by $\chi_E$. We shall use $\CcR$ to represent the space of smooth functions of compact support and $\SchwR$ for the Schwartz space of smooth rapidly decreasing functions (with obvious adjustments for spaces in higher dimensions). The commutator of two operators $A$, $B$ is denoted by $[A,B]:=AB - BA$.

In the one-dimensional setting, we will refer  to the first and second order differential operators with $\Ntp$ and $\Dtp$, respectively, reserving the symbols $\nabla$ and $\Delta$ for statements in higher dimensions. When the relevant differentiation variable is time, we shall use $\Ntime$ and $\Dtime$ for the first and second order derivatives, respectively.

If $\cH$ denotes a Hilbert space, we shall use $\langle \cdot, \cdot \rangle_{\cH}$ and $\| \cdot \|_{\cH}$ to represent the inner product and norm on that space. The $L^2$ inner product shall be denoted by $\langle \cdot, \cdot \rangle_2$, or just by $\langle \cdot, \cdot \rangle$ if there is no ambiguity, and the $L^2$ norm by $\| \cdot \|_2$ or just by $\| \cdot \|$. The other $L^p$ norms will be represented by $\| \cdot \|_p$ with $L^{\infty}$ denoting the space of essentially bounded functions endowed with the essential $\sup$ norm $\| \cdot \|_{\infty}$.

Let $\emptyset \neq \Omega \subset \Rd$ be open, $k \in \N$ and $p \in [1, +\infty]$. We will denote the Sobolev spaces by $W^{k,p}(\Omega)$ and $W^{k,p}_0(\Omega)$ (the latter representing as usual the closure of $\CcOm$ in $W^{k,p}(\Omega)$, see \eg~\cite[Sub-sec.~V.3]{EE} for definitions). We shall generally be concerned with the particular cases where $\Omega = \R$, $k = 1 \text{ or } 2$ and $p = 2$.

If $\cB_1, \cB_2$ are two Banach spaces, $\sL(\cB_1, \cB_2)$ shall denote the (Banach) space of bounded linear operators from $\cB_1$ to $\cB_2$. As it is customary, if $\cB$ is a Banach space, $\sL(\cB)$ means $\sL(\cB, \cB)$. If the operator $\opT \in \sL(\cB)$, then $\rad(\opT)$ represents its spectral radius, i.e. $\rad(\opT) := \sup\{|z|: z \in \sigma(\opT)\}$ with $\sigma(\opT)$ denoting the spectrum of $\opT$. Unless otherwise stated, for a closed, densely defined linear operator $\opT$ on a Banach space $\cB$, we will use $\sigma_{e2}(\opT)$ to denote the essential spectrum of $\opT$ as determined using singular sequences (see e.g. \cite[Thm.~IX.1.3]{EE}), a closed subset of $\C$. As usual, $\sigma_{p}(\opT)$ will denote the set of eigenvalues of $\opT$ and $\rho(\opT)$ its resolvent set.

If $\opH$ and $\opH_1$ are two linear operators acting in the Hilbert space $\cH$, we say that $\opH_1$ is an \textit{extension} of $\opH$, and write $\opH_1 \supset \opH$, if $\Dom(\opH_1) \supset \Dom(\opH)$ and $\opH_1 u = \opH u$ for all $u \in \Dom(\opH)$. Note that our notation covers the case $\Dom(\opH_1) = \Dom(\opH)$, \ie~the extension does \textit{not} have to be proper.

If $\cH_1$ and $\cH_2$ represent two Hilbert spaces, we will denote by $\cH_1 \oplus \cH_2$ the product space endowed with the inner product
\begin{equation*}
	\langle u, v \rangle_{\cH_1 \oplus \cH_2} := \langle u_1, v_1 \rangle_{\cH_1} + \langle u_2, v_2 \rangle_{\cH_2}, \quad u_1, v_1 \in \cH_1, \quad u_2, v_2 \in \cH_2,
\end{equation*}
which is also Hilbert, and $\| \cdot \|_{\cH_1 \oplus \cH_2} := \langle \cdot, \cdot \rangle_{\cH_1 \oplus \cH_2}^{\frac12}$ will represent the associated norm.

To avoid introducing multiple constants whose exact value is inessential for our purposes, we write $a \lesssim b$ to indicate that, given $a,b \ge 0$, there exists a constant $C>0$, independent of any relevant variable or parameter, such that $a \le Cb$. The relation $a \gtrsim b$ is defined analogously whereas $a \approx b$ means that $a \lesssim b$ \textit{and} $a \gtrsim b$. When it becomes relevant to underlie the dependency of an implicit constant on one or more parameters, $p_1, p_2, \dots$, we will use the notation $\ls_{p_1, p_2, \dots}$, $\gs_{p_1, p_2, \dots}$ or $\approx_{p_1, p_2, \dots}$, as appropriate. We shall use $\BigO_{p_1, p_2, \dots}$ (big-O notation) with a similar meaning.

In the rest of this section, we summarise the key properties of the damped wave equation and the main tools relied upon in the paper.

\subsection{Fourier transform and pseudo-differential operators}
For $u \in \SchwR$, the Fourier and inverse Fourier transforms read (with $x,\xi \in \R$)
\begin{align*}
	\sF u(\xi) := \intR e^{-i \xi x} u(x) \odd x, \qquad \sF^{-1} u(x) := \intR e^{i x \xi} u(\xi) \odd \xi, \quad \odd \cdot : = \frac{\dd \cdot}{\sqrt{2\pi}};
\end{align*}
we also use $\hat{u} := \sF u \text{ and } \check{u} := \sF^{-1} u$, and retain the same notations to refer to the corresponding isometric extensions to $\Lt(\R)$.

We recall that the Schwartz space, $\SchwR$, is endowed with the family of semi-norms
\begin{equation*}
	|f|_{k,\sS} := \underset{\alpha + \beta \le k}{\max} \, \underset{x \in \R}{\sup} \, \langle x \rangle^{\alpha} \, |\partial^{\beta}_{x} f(x)|, \quad k \in \N_0.
\end{equation*}

Given $m \in \R$, the symbol class $\cS^m_{1,0}(\R \times \R)$ is the vector space of smooth functions $p: \R \times \R \rightarrow \C$ such that for any $\alpha, \; \beta \in \N_0$ there exists $C_{\alpha,\beta} > 0$ satisfying
\begin{equation}
	\label{eq:symbolclass}
	|\partial^{\alpha}_{\xi}\partial^{\beta}_{x} p(\xi, x)| \le C_{\alpha,\beta} \, \langle x \rangle^{m - \beta}, \quad (\xi,x) \in \R \times \R.
\end{equation}
This space is endowed with a natural family of semi-norms defined by
\begin{equation}
	\label{eq.symbolclass.seminorm}
	|p|_k^{(m)} := \underset{\alpha, \beta \le k}{\max} \, \underset{\xi, x \in \R}{\sup} \, \langle x \rangle^{-m+\beta} \, |\partial^{\alpha}_{\xi}\partial^{\beta}_{x} p(\xi, x)|, \quad k \in \N_0.
\end{equation}

We associate a pseudo-differential operator with the symbol $p \in \cS^m_{1,0}(\R \times \R)$ via
\begin{equation*}
	\opOP(p) u(\xi) := \intR e^{-i \xi x} p(\xi,x) \check{u}(x) \odd x,  \quad \xi \in \R, \quad u \in \SchwR,
\end{equation*}
and it can be shown that this is a bounded mapping on $\SchwR$ (see\cite[Thm.~3.6]{Abels-2011}).

The following result will be used later on and we include it here for convenience. We refer to \cite[Lem.~4.4]{ArSi-resolvent-2022} for a proof.

\begin{lemma}
	\label{lem:pdo.comp}
	Let $F \in \CiR$ and $m > 0$ be such that 
	\begin{equation}
		\label{eq:Vsymbolclass}
		\forall n \in \N_0, \quad \exists C_n>0, \quad |F^{(n)}(x)| \le C_n \langle x \rangle^{m - n}, \quad x \in \R,
	\end{equation}
	and let $\phi \in \CiR \cap \LiR$ be such that $\supp \phi'$ is bounded. For $j \in \N_0$ and $u \in \SchwR$, we define the operators (with $\opP:=\opP^{(0)}$ and $\opQ:=\opQ^{(0)}$)
	\begin{align}
		\label{eq:composition.pq.def}
		\opP^{(j)} u := \sF F^{(j)} \sF^{-1} u, \qquad \opQ^{(j)} u := \phi^{(j)} u.
	\end{align}
	Then, for any $N \in \N_0$, we have 
	\begin{equation}
		\label{eq:compositionformula}
		[\opP,\opQ]  u = \sum_{j=1}^{N} \frac{i^j}{j!} \opQ^{(j)} \opP^{(j)} u + \opR_{N+1} u, \quad u \in \SchwR,
	\end{equation}
	where $R_{N+1}$ is a pseudo-differential operator with symbol $r_{N+1} \in \cS^{m-N-1}_{1,0}(\R \times \R)$
	\begin{equation}
		\label{eq:comp.operator.remainder}
		\opR_{N+1} \, u(\xi) := \intR e^{-i \xi x} r_{N+1}(\xi,x) \check{u}(x) \odd x.
	\end{equation}
	Moreover, for every $N \in \N$ with $N > m$, there exist $l = l(N) \in \N$ and $K_N>0$, independent of $F$ and $\phi$, such that %
	\begin{equation}
		\label{eq:R_N.est}
		\| \opR_{N+1} u \| \le K_N  \underset{0 \le j \le l}{\max} \left\{ \| \phi^{(N + 1 + j)}\|_\infty \right\} \| u \|. 
	\end{equation}
\end{lemma}

\subsection{Schr\"odinger operators with complex potentials}
\label{ssec:Schr.prelim}
Let $\emptyset \neq \Omega \subset \Rd$ be open. For a measurable function $m: \Omega \to \C$, we denote the maximal domain of the multiplication operator determined by the function $m$ as
\begin{equation}\label{eq:multiplication.domain}
\Dom(m) = \{ u \in L^2(\Omega) \, : \, mu \in L^2(\Omega)\};
\end{equation}	
the Dirichlet Laplacian in $L^2(\Omega)$ is denoted by $-\Delta_D$ and 
\begin{equation}
\Dom(\Delta_D) = \{ u \in W_0^{1,2}(\Omega) : \, \Delta u \in L^2(\Omega)\}.		
\end{equation}

Suppose that the complex potential $V:\Omega \rightarrow \C$, $V = V_{u} + V_{b}$, satisfies $\Re V \geq 0$, $V_{u} \in C^1\left(\overline{\Omega}\right)$, $V_{b} \in \LiOm$ and, 
with $\eps_{\rm crit} = 2-\sqrt{2}$,
\begin{equation}\label{eq:vgrowth}
\exists\eps_\nabla \in [0,\eps_{\rm crit}), \quad  \exists M_{\nabla} \geq 0, \quad  
|\nabla V_{u}| \leq \eps_\nabla |V_{u}|^\frac 32 + M_{\nabla} \quad \text{a.e.~in~} \Omega.
\end{equation}

Under these assumptions on $V$ one can find the (Dirichlet) m-accretive realization $\opH = -\Delta_D +V$, with $\Dom(\opH) = \{ u \in W_0^{1,2}(\Omega) \cap \Dom(|V|^\frac12) : \; (-\Delta + V) u \in L^2(\Omega)\}$, by appealing to a generalised Lax-Milgram theorem \cite[Thm.~2.2]{Almog-2015-40}. It is also known that the domain and the graph norm of $\opH$ separate, \ie~$\Dom(H) = \Dom(\Delta_D) \cap \Dom(V)$ and
\begin{equation}
\label{eq:Hcorelowerbound}
\| \opH u \|^2 + \| u \|^2 \gs  \| \Delta_D u \|^2 + \| V u \|^2 + \|u\|^2, \quad u \in \Dom(\opH).
\end{equation}
Furthermore, 
\begin{equation}
\label{eq:Habstractcore}
\Core := \{ u \in \Dom(\opH): \; \supp u \text{ is bounded}\}
\end{equation}
is a core of $\opH$. 
For details, see \cite{Almog-2015-40, Krejcirik-2017-221, Semoradova-toappear} and, for cases with minimal regularity of $V$, see \cite{Brezis-1979-58, Kato-1978-5}, \cite[Chap.~VI.2]{EE}.

\subsection{Generalised Airy operators}
\label{ssec:Airy.prelim}
In Section~\ref{sec:quadop}, we use operators in $L^2(\R)$ of type (with $a \in \LilocR$, $a \ge 0$ a.e. and $\underset{|x| \ge N}{\essinf} a(x) \to +\infty$ as $N \to +\infty$)
\begin{equation}
	\label{eq:Airy.def}
	\opA = \Nt + a(x), \quad \Dom(\opA) = W^{1,2}(\R) \cap \Dom(a),
\end{equation}
which we refer to as generalised Airy operators (on Fourier space). The adjoint operator is
\begin{equation*}
	\opA^* = \Ntp + a(x), \quad \Dom(\opA^*) = W^{1,2}(\R) \cap \Dom(a),
\end{equation*}
and many properties of the usual complex Airy operators are preserved for $\opA$ and $\opA^*$. Namely they have compact resolvent, empty spectrum and
\begin{equation}
	\label{eq:Abeta.graphnorm.realline.def}
	\begin{aligned}
		\| \opA u \|^2 + \| u \|^2 &\gs \|  u' \|^2 + \left\| a u \right\|^2 + \| u \|^2, \quad u \in \Dom(\opA),\\
		\|  \opA^* u \|^2 + \| u \|^2 &\gs \| u' \|^2 + \left\| a u \right\|^2 + \| u \|^2, \quad u \in \Dom(\opA^*),
	\end{aligned}
\end{equation}
where the domain and graph norm separation require the additional assumptions that $a \in \LilocR \cap C^1\left( \R\setminus [-x_0,x_0] \right)$, with some $x_0 > 0$, and that there exist $\eps \in (0, 1)$ and $M > 0$ such that
\begin{equation*}
	|a'(x)| \le \eps (a(x))^2 + M, \quad |x| > x_0,
\end{equation*}
see \cite[App.~A]{ArSi-resolvent-2022} for details and \cite{ArSi-generalised-2022} for resolvent norm estimates.

\subsection{The damped wave equation}
\label{ssec:dwe.prelim}
The focus of our study shall be the linear operator $\opG$ (see below) associated with the one-dimensional DWE represented by \eqref{eq:dwe.2ndorder}. Following a standard procedure, we re-write the problem as a first order system of linear equations
\begin{equation*}
	\left\{
		\begin{aligned}
			\Ntime u_1(t,x) &= u_2(t,x),\\
			\Ntime u_2(t,x) &= (\Dtp - q(x)) u_1(t,x) - 2 a(x) u_2(t,x),
		\end{aligned}
	\right.
\end{equation*}
which leads naturally to the formal operator matrix
\begin{equation*}
	\begin{pmatrix}
		0 & \opI \\
		\Dtp - q & -2 a
	\end{pmatrix}
	.
\end{equation*}

In order to properly define such a matrix as an unbounded (non-self-adjoint) operator, we follow Section~4 of \cite{gerhat2022schur} which specialises a new general framework for the spectral analysis of operator matrices to the particular case of the DWE. Assuming that $a, q \in L^1_{\loc}(\R)$ with $a, q \ge 0$ a.e., let $\cH_1 := \cW(\R)$ represent the completion of $\CcR$ with respect to the inner product
\begin{equation*}
	\langle f, g \rangle_{\cW} := \intR \Ntp f(x) \overline{\Ntp g(x)} \dd x + \intR q(x) f(x) \overline{g(x)} \dd x
\end{equation*}
and let $\cH_2 := \Lt(\R)$. Furthermore, define the Hilbert space
\begin{equation*}
	\begin{aligned}
		\cD_S &:= W^{1,2}(\R) \cap \Dom(q^\frac12) \cap \Dom(a^\frac12),\\
		\langle f, g \rangle_S &:= \intR \Ntp f(x) \overline{\Ntp g(x)} \dd x + \intR q(x) f(x) \overline{g(x)} \dd x + \intR a(x) f(x) \overline{g(x)} \dd x\\
		&\quad + \intR f(x) \overline{g(x)} \dd x, \quad f, g \in \cD_S,
	\end{aligned}
\end{equation*}
and let $\cD_S^*$ be the space of bounded, conjugate-linear functionals on $\cD_S$. It can be shown that the canonical embeddings $\cD_S \subset \cH_2 \subset \cD_S^*$ are continuous with dense range, that $\CcR$ is densely contained in $\cD_S$ and that $\cD_S$ can also be continuously embedded in $\cH_1$ (see \cite[Prop.~4.6]{gerhat2022schur}). Moreover, the operators
\begin{align*}
	&0 \in \sL(\cH_1), &\opI \in \sL(\cD_S, \cH_1),\\
	&\Dtp - q \in \sL(\cH_1, \cD_S^*), &-2 a \in \sL(\cD_S, \cD_S^*),
\end{align*}
with $\Dtp - q$ and $a$ the unique extensions of
\begin{equation*}
	\begin{aligned}
		((\Dtp - q) f, g)_{\cD_S^* \times \cD_S} &:= -\intR \left(\Ntp f(x) \overline{\Ntp g(x)} + q(x) f(x) \overline{g(x)}\right) \dd x,\\
		(a f, g)_{\cD_S^* \times \cD_S} &:= \intR a(x) f(x) \overline{g(x)} \dd x,
	\end{aligned}
	\quad f, g \in \CcR,
\end{equation*}
are well-defined (see \cite[Prop.~4.9]{gerhat2022schur}). We are therefore in a position to introduce the operator matrix
\begin{equation}
	\label{eq:Ghat.def}
	\widehat\opG :=
	\begin{pmatrix}
		0 & \opI \\
		\Dtp - q & -2 a
	\end{pmatrix}
	\in \sL(\cH_1 \oplus \cD_S, \cH_1 \oplus \cD_S^*),
\end{equation}
its (second) Schur complement
\begin{equation*}
	\widehat \opS(\la) := -2 a - \la + \frac1{\la} (\Dtp - q)|_{\cD_S} \in \sL(\cD_S, \cD_S^*), \quad \la \in \C \setminus \{0\},
\end{equation*}
and the corresponding restrictions
\begin{equation}
	\label{eq:G.def}
	\begin{aligned}
		\opG &:= \widehat\opG|_{\Dom(\opG)},\\
		\Dom(\opG) &:= \{u \in \cH_1 \oplus \cD_S: \widehat\opG u \in \cH_1 \oplus \cH_2\}\\
		&\;= \{u := (u_1, u_2)^t \in \cH_1 \oplus \cD_S: (\Dtp - q) u_1 -2 a u_2 \in \cH_2\},
	\end{aligned}
\end{equation}
and
\begin{equation*}
	\begin{aligned}
		\opS(\la) &:= \widehat\opS(\la)|_{\Dom(\opS(\la))},\\
		\Dom(\opS(\la)) &:= \{u \in \cD_S: \widehat\opS(\la) u \in \cH_2\}\\
		&\;= \{u \in \cD_S: (\Dtp - q - 2 \la a) u \in \cH_2\}.
	\end{aligned}
\end{equation*}
The fundamental result derived from the general setting outlined above is that the operator $\opG$ is m-dissipative with dense domain in both $\cH_1 \oplus \cD_S$ and $\cH_1 \oplus \cH_2$ and it therefore generates a $C_0$-semigroup of contractions on $\cH_1 \oplus \cH_2$. Furthermore, for all $\la \in \C \setminus (-\infty, 0]$, it can also be shown that $\Dom(\opS(\la))$ is dense in $\cD_S$ and that $\opG$ and $\opS(\la)$ are spectrally equivalent in the following sense
\begin{equation}
	\label{eq:G.Tla.spectral.equivalence}
	\la \in \sigma(\opG) \iff 0 \in \sigma(\opS(\la))
\end{equation}
(see \cite[Lem.~4.13, Thm.~4.2]{gerhat2022schur}). Moreover, if $\la \in \C \setminus (-\infty, 0]$ and $0 \in \rho(\opS(\la))$, the operator matrix (with $\opI := \opI_{\cH_1 \to \cH_1}$)
\begin{equation}
	\label{eq:Rla.def}
	\opR_\la :=
	\begin{pmatrix}
		-\frac1{\la} \opI + \frac1{\la^2} \widehat\opS(\la)^{-1} (\Dtp - q) & \frac1{\la} \opS(\la)^{-1}\\
		\frac1{\la} \widehat\opS(\la)^{-1} (\Dtp - q) & \opS(\la)^{-1}
	\end{pmatrix}
	\in \sL(\cH_1 \oplus \cH_2)
\end{equation}
is both a left and right inverse for $\opG - \la$ (see the proof of \cite[Thm.~2.8]{gerhat2022schur} for details).

Letting $\opH_q$ denote the Friedrichs extension of $\Dt + q$ initially defined on $\CcR$, i.e.
\begin{equation}
	\label{eq:Hq.def}
	\opH_q := \Dt + q, \quad \Dom(\opH_q) := \{u \in W^{1,2}(\R) \cap \Dom(q^\frac12) : (\Dt + q) u \in L^2(\R)\},
\end{equation}
it has been proven in \cite{Freitas-2018-264} under more restrictive assumptions on the damping and potential functions (which hold for the operators covered in this paper, see Remark~\ref{rmk:domain.separation}) that the domain of the quadratic operator function in $\Lt(\R)$
\begin{equation}
	\label{eq:Tla.def}
	\opT(\la) := -\la \opS(\la) = \opH_q + 2 \la a + \la^2, \quad \la \in \C \setminus (-\infty, 0],
\end{equation}
separates and does not depend on $\la$
\begin{equation}
	\label{eq:Tla.domain}
	\Dom(\opT(\la)) = \Dom(\opH_q) \cap \Dom(a) \subset \cH_1 \cap \cD_S;
\end{equation}
moreover, the subspace
\begin{equation}
	\label{eq:Tla.core}
	\Dcore := \{ u \in \Dom(\opH_q): \; \supp u \text{ is compact in } \R \} \subset \Dom(a)
\end{equation}
is a core for $\opT(\la)$ and $\opT(\la)^* = \opT(\overline{\la})$, for $\la \in \C \setminus (-\infty, 0]$ (see \cite[Thm.~2.4]{Freitas-2018-264}). It also holds true that, if the damping $a$ satisfies \cite[Asm.~I]{Freitas-2018-264} and is unbounded (see \cite[Asm.~II]{Freitas-2018-264}), the set $\sigma(\opG) \cap \C \setminus (-\infty, 0]$ consists of at most a countable set of isolated eigenvalues of finite multiplicity which may only accumulate at $(-\infty, 0]$ (see \cite[Thm.~3.2]{Freitas-2018-264}).

\section{Assumptions and statement of the main result}
\label{sec:assumptions}
We begin by presenting the assumptions that $a$ and $q$ will obey throughout the rest of the paper. We shall follow the notation introduced in Sub-section~\ref{ssec:dwe.prelim}.

\begin{asm-sec}
\label{asm:a.q.dwe}
Let $a, q \in \CiR$ such that $a \ge 0, q \ge 0$ and assume that the following conditions are satisfied for some $x_0 \in \Rplus$:
\begin{enumerate} [\upshape (i)]
	\item \label{itm:a.incr.unbd} $a$ is unbounded:
	\begin{equation}
		\label{eq:a.unbd.incr}
		\lim_{|x| \to +\infty} a(x) = +\infty;
	\end{equation}
	\item \label{itm:a.symbolclass} $a$ has controlled derivatives:
	\begin{equation}
		\label{eq:a.symb}
		\forall n \in \N, \quad \exists C_n>0, \quad |a^{(n)}(x)| \le C_n \, \left(1 + a(x)\right) \, \langle x \rangle^{- n}, \quad x \in \R;
	\end{equation}
	\item \label{itm:q.symbolclass} $q$ has controlled derivatives:
	\begin{equation}
		\label{eq:q.symb}
		\forall n \in \N, \quad \exists C'_n>0, \quad |q^{(n)}(x)| \le C'_n \, \left(1 + q(x)\right) \, \langle x \rangle^{- n}, \quad x \in \R;
	\end{equation}
	\item \label{itm:a.gt.q} $q$ is eventually not bigger than $a$:
	\begin{equation}
		\label{eq:a.gt.q}
		\exists K > 0, \quad q(x) \le K a(x), \quad |x| > x_0.
	\end{equation}
\end{enumerate}
\end{asm-sec}

\begin{example}
	Damping functions satisfying Assumption~\ref{asm:a.q.dwe}~\ref{itm:a.incr.unbd}-\ref{itm:a.symbolclass} include $a(x) = x^{2n}$, $n \in \N$, $a(x) = \langle x \rangle^p$, $p > 0$, and $a(x) = \log \langle x \rangle^p$, $p > 0$. The same functions are valid potentials $q(x)$ in addition to smooth, non-negative, bounded functions such as $q(x) = k$, $k \ge 0$, and $q(x) = \langle x \rangle^p$, $p \le 0$.
\end{example}

\begin{remark}
	\label{rmk:a.q.symbolclass}
	It can be shown using Assumption~\ref{asm:a.q.dwe}~\ref{itm:a.symbolclass} with $n = 1$ that there exists $m_a > 0$ such that
	\begin{equation*}
		a(x) \ls \langle x \rangle^{m_a}, \quad x \in \R,
	\end{equation*}
	(e.g. see the comments following Example 3.1 in \cite{Krejcirik-2019-276}). Furthermore, for any $n \in \N$
	\begin{equation*}
		|a^{(n)}(x)| \ls (1 + a(x)) \langle x \rangle^{-n} \ls \langle x \rangle^{m_a - n}, \quad x \in \R,
	\end{equation*}
	which shows that $a \in \cS^{m_a}_{1,0}(\R \times \R)$. Similarly it follows from Assumption~\ref{asm:a.q.dwe}~\ref{itm:q.symbolclass} that there exists $m_q > 0$ such that $q \in \cS^{m_q}_{1,0}(\R \times \R)$.
\end{remark}

\begin{remark}
	\label{rmk:domain.separation}
	If $a$ and $q$ satisfy Assumption~\ref{asm:a.q.dwe}, then they automatically obey Assumptions~I and II in \cite{Freitas-2018-264} with $\Omega = \R$ and $a_s = 0$. Therefore the properties of $\opG$ and $\opT(\la)$ described in Sub-section~\ref{ssec:dwe.prelim} hold, in particular the domain separation $\Dom(\opT(\la)) = \Dom(\opH_q) \cap \Dom(a)$. Furthermore, we also have $\Dom(\opH_q) = \WttR \cap \Dom(q)$ (refer to Sub-section~\ref{ssec:Schr.prelim}). It therefore follows from Assumption~\ref{asm:a.q.dwe}~\ref{itm:a.gt.q} that $\Dom(\opT(\la)) = \WttR \cap \Dom(a)$ and \eqref{eq:Tla.core} simplifies to
	\begin{equation*}
		\Dcore = \{ u \in \WttR: \supp u \text{ is compact in } \R \}.
	\end{equation*}
\end{remark}

We now state our main result regarding the asymptotic behaviour of the norm of the resolvent of $\opG$ in the left-hand side (with respect to the imaginary axis) of the complex plane.

\begin{theorem}
	\label{thm:resolvent.G}
	Let $a$ and $q$ satisfy Assumption~\ref{asm:a.q.dwe} and let $\opG$ be the linear operator \eqref{eq:Ghat.def}-\eqref{eq:G.def} acting in $\cH := \cH_1 \oplus \cH_2$, with $\cH_1$ and $\cH_2$ as defined in Sub-section~\ref{ssec:dwe.prelim}. Let $K \subset \overline \Rplus$ be a bounded subset and $\la := -c + i b \in \C$ with $c \in K$ and $b \in \R \setminus \{0\}$. Then as $|b| \rightarrow +\infty$
	\begin{equation}
		\label{eq:resnorm.G}
		\|(\opG - \lambda)^{-1} \| \approx_K 1.
	\end{equation}
\end{theorem}

\begin{remark}
	\label{rmk:b.asymptotic}
	The statement of Theorem~\ref{thm:resolvent.G} describes the asymptotic behaviour of the resolvent of $\opG$ as a function of the spectral parameter $b$ and it should be understood as follows: there exists $b_0(K) > 0$ such that for all $|b| \ge b_0(K)$ and all $c \in K$, then \eqref{eq:resnorm.G} holds. The same remark applies to other asymptotic results involving $\la = -c + i b$ (whether in relation to $\opG$ or the quadratic family $\opT(\la)$) throughout this paper.
\end{remark}

\begin{remark}
	\label{rmk:rays.growth}
	We note that the statement \eqref{eq:resnorm.G} is far from obvious. For example, it has been shown (see \cite[Ex.~3.9]{Arifoski-2020-52}) that for polynomial-like dampings and potentials, i.e. for functions $a,q \in C^{n+1}(\R)$, with $n > 1$, satisfying
	\begin{equation*}
		\forall x \gs 1, \quad a(x) = x^p, \quad |q^{(j)}(x)| \ls x^{r-j}, \quad p,r \in \Rplus, \quad 0 \le j \le n, \quad j \in \N_0,
	\end{equation*}
	the norm of the resolvent of the corresponding generator $\opG$ diverges to $+\infty$ along \textit{any} ray in the second (or third) quadrant
	\begin{equation*}
		\| (\opG - \la)^{-1} \| \gs b^{(n-1)(p+1) + 2}, \quad b \to +\infty,
	\end{equation*}
	where $\la = -\alpha + i \beta$, $\alpha = a(b)$, $\beta = k \alpha$, and $k \in \Rplus$ is arbitrary. A similar divergence (albeit with a different rate: $(\log b)^{n-1} b^{n+1}$) is observed for logarithmic dampings and potentials (see \cite[Ex.~3.11]{Arifoski-2020-52}). By adding obvious restrictions, both sets of examples can be chosen so that they fall within the scope of Assumption~\ref{asm:a.q.dwe} and therefore Theorem~\ref{thm:resolvent.G} applies to them, meaning that $\| (\opG - \la)^{-1} \|$ is (asymptotically) approximately constant on vertical lines.
\end{remark}

We close this section by drawing some consequences from the theorem that highlight the dependency of the long-time behaviour of the corresponding semigroup on the location of $\sigma(\opG)$.

\begin{lemma}
	\label{lem:sigmap.G.not.imaginary}
	Let $a$, $q$, $\opG$ and $\cH$ be as in the statement of Theorem~\ref{thm:resolvent.G} and assume furthermore that $a \ne 0$ almost everywhere. Then $\sigma_p(G) \cap i \R = \emptyset$.
\end{lemma}
\begin{proof}
	Assume firstly that there exists $0 \ne u:= (u_1, u_2)^t \in \Dom(\opG) \subset \cW(\R) \oplus \cD_S$ (see \eqref{eq:G.def}) such that $u \in \Ker(\opG)$. It follows
	\begin{equation*}
		\| \opG u \|_{\cH} = 0 \implies \| \Ntp u_2 \|^2 + \| q^{\frac12} u_2 \|^2 + \| \opH_q u_1 + 2 a u_2 \|^2 = 0
	\end{equation*}
	and hence we have $\| \Ntp u_2 \| = 0$. Since $u_2 \in \cD_S \subset \WotR \subset \Lt(\R)$, we obtain that $u_2 = 0$  and therefore $a u_2 = 0$. This shows that $\opH_q u_1 \in \Lt(\R)$ and $\opH_q u_1 = 0$ and consequently for any $f \in \CcR$
	\begin{equation*}
		0 = \langle \opH_q u_1, f \rangle = \langle \Ntp u_1, \Ntp f \rangle + \langle q^{\frac12} u_1, q^{\frac12} f \rangle = \langle u_1, f \rangle_{\cW}.
	\end{equation*}
	Noting that $\CcR$ is dense in $\cW(\R)$ for $\| \cdot \|_{\cW}$, we conclude that $u_1 = 0$ and hence $0 \notin \sigma_p(\opG)$.
	
	Let $\la := i b$, with $b \in \R \setminus \{0\}$, and assume that there exists $0 \ne u:= (u_1, u_2)^t \in \Dom(\opG)$ such that $u \in \Ker(\opG - \la)$. By Claim (ii) in the proof of \cite[Thm.~3.2]{Freitas-2018-264} (see Remark~\ref{rmk:domain.separation}), we deduce that $\la u_1 = u_2$ and $u_2 \in \Ker(\opT(\la))$. Then
	\begin{equation*}
		\begin{aligned}
			\opT(\la) u_2 = 0 &\implies \langle \opH_q u_2, u_2 \rangle + 2 \la \langle a u_2, u_2 \rangle + \la^2 \langle u_2, u_2 \rangle = 0\\
			&\implies \| \Ntp u_2 \|^2 + \| q^{\frac12} u_2 \|^2 + 2 i b \| a^{\frac12} u_2 \|^2 - b^2 \| u_2 \|^2 = 0
		\end{aligned}
	\end{equation*}
	and hence (note $b \ne 0$) we have $\| a^{\frac12} u_2 \| = 0$. Since $a > 0$ a.e. by assumption, we conclude that $u_2 = 0$ and therefore $u_1 = 0$, which completes the proof.
\end{proof}

We recall some definitions and properties related to semigroups. The spectral bound of a linear operator $\opA$ is given by
	\begin{equation}
	\label{eq:spec.bound.def}
	s(\opA) := \sup\{\Re \la: \la \in \sigma(\opA)\}.
\end{equation}
If $(\opS_t)_{t \ge 0}$ is a $C_0$-semigroup acting on a Banach space, we define its growth bound as
\begin{equation}
	\label{eq:growth.bound.def}
	\omega_0 := \inf\{\omega \in \R: \exists M_{\omega} \ge 1 \text{ s.t. } \| \opS_t \| \le M_{\omega} e^{w t}, \;\; \forall t \ge 0\}.
\end{equation}
The following general relation holds between the growth bound of a $C_0$-semigroup $(\opS_t)_{t \ge 0}$ on a Banach space and the spectral bound of its generator $\opA$
\begin{equation}
	\label{eq:sA.le.omega0}
	-\infty \le s(\opA) \le \omega_0 < +\infty
\end{equation}
(see \cite[Cor.~II.1.13]{Engel-Nagel-book}). Lastly the growth bound of a $C_0$-semigroup on a Hilbert space with generator $\opA$ is given by
\begin{equation}
	\label{eq:growth.bound.H}
	\omega_0 = \inf\{\omega > s(\opA): \underset{s \in \R}{\sup} \|(\opA - (\omega + i s))^{-1}\| < +\infty\}
\end{equation}
(see \cite[Ex.~V.1.13]{Engel-Nagel-book} or the proof of Theorem~2.3 in \cite{Freitas-2020-148}).

Our next result shows that, with additional conditions on $a, q$, the semigroup generated by the operator $\opG$ in Theorem~\ref{thm:resolvent.G} is uniformly exponentially stable and therefore the solutions of the corresponding abstract Cauchy problem decay exponentially as $t \to +\infty$.

\begin{corollary}
	\label{cor:sg.G.decay}
	Let the assumptions of Theorem~\ref{thm:resolvent.G} hold and assume furthermore that $a \ne 0$ almost everywhere and
	\begin{equation*}
		\exists K' > 0, \quad q(x) \ge K' a(x), \quad |x| > x_0.
	\end{equation*}
	Then we have
	\begin{equation}
		\label{eq:omega0.eq.sG}
		\omega_0 = s(\opG) < 0.
	\end{equation}
\end{corollary}
\begin{proof}
	Under the assumptions of the corollary and applying \cite[Rmk.~3.3]{Freitas-2018-264}, there exists $\alpha_q > 0$ such that the spectral equivalence \eqref{eq:G.Tla.spectral.equivalence} holds for $\la \in \C \setminus (-\infty, -\alpha_q]$ and moreover $\sigma(\opG) \setminus (-\infty, -\alpha_q]$ consists of eigenvalues with finite multiplicity which may only accumulate at points in $(-\infty, -\alpha_q]$. This last observation combined with Lemma~\ref{lem:sigmap.G.not.imaginary} shows that $s(\opG) < 0$.
	
	Moreover note that
	\begin{equation}
		\label{eq:sup.res.G.sG}
		\underset{s \in \R}{\sup} \|(\opG - (\omega + i s))^{-1}\| < +\infty,\quad \forall \omega > s(\opG).
	\end{equation}
	If $\omega > 0$, the above claim is a consequence of the fact that $\opG$ is m-dissipative, whereas for $s(\opG) < \omega \le 0$, it follows from Theorem~\ref{thm:resolvent.G}. Using \eqref{eq:sup.res.G.sG} and \eqref{eq:growth.bound.H}, we deduce that $\omega_0 = s(\opG)$, as required.
\end{proof}

\section{The associated quadratic operator}
\label{sec:quadop}
In this section, our aim is to formulate a number of properties of the operator family $\opT(\la)$ introduced in Sub-section~\ref{ssec:dwe.prelim} which will help us prove Theorem~\ref{thm:resolvent.G} in Section~\ref{sec:mainthmG}. We begin by studying the behaviour of $\| \opT(\la)^{-1} \|$ on the positive real axis.

\begin{proposition}
	\label{prop:Tmu.estimates}
	Let $a$ and $q$ satisfy Assumption~\ref{asm:a.q.dwe}, $\opH_q$ be as in \eqref{eq:Hq.def} and $\mu > 0$. Define $\opT(\mu)$ as in \eqref{eq:Tla.def}-\eqref{eq:Tla.domain}, i.e.
	\begin{equation*}
		\opT(\mu) := \opH_q + 2 \mu a + \mu^2, \quad \Dom(\opT(\mu)) := \WttR \cap \Dom(a).
	\end{equation*}
	Then for every $u \in \Dom(\opT(\mu))$, we have
	\begin{equation}
		\label{eq:Tmu.graphnorm}
		\| u'' \|^2 + \| q u \|^2 + \mu^2 \| a u \|^2 \ls \| \opT(\mu) u \|^2 + \mu^2 \| u \|^2, \quad \mu \to +\infty.
	\end{equation}
	Furthermore, the following inequalities hold
	\begin{gather}
		\label{eq:Tmuminus1.norm}\| \opT(\mu)^{-1} \| \le \mu^{-2}, \quad \mu > 0,\\
		\label{eq:Hq.Tmuminus1.norm}\| \opH_q \opT(\mu)^{-1} \| + \| \opT(\mu)^{-1} \opH_q \| \ls 1, \quad \mu \to +\infty,\\
		\label{eq:Hq12.Tmuminus1.norm}\| \opH^{\frac12}_q \opT(\mu)^{-1} \| + \| \opT(\mu)^{-1} \opH^{\frac12}_q \| \ls \mu^{-1}, \quad \mu \to +\infty,\\
		\label{eq:Hq12.Tmuminus12.norm}\| \opH^{\frac12}_q \opT(\mu)^{-\frac12} \| + \| \opT(\mu)^{-\frac12} \opH^{\frac12}_q \| \ls 1, \quad \mu > 0.
	\end{gather}
\end{proposition}
\begin{proof}
	Let $u \in \Dom(\opT(\mu))$ and $\opH_{\mu} := \opH_q + 2 \mu a$ with $\Dom(\opH_{\mu}) := \Dom(\opT(\mu))$. Then $\opT(\mu)$ and $\opH_{\mu}$ are non-negative, self-adjoint operators and we have
	\begin{equation}
		\label{eq:Tmu.lowbound}
		\| \opT(\mu) u \|^2 = \| \opH_{\mu} u \|^2 + \mu^4 \| u \|^2 + 2 \mu^2 \langle \opH_{\mu} u, u\rangle \ge \| \opH_{\mu} u \|^2.
	\end{equation}
	Furthermore using $q + 2 \mu a \ge 0$
	\begin{equation}
		\label{eq:Hq.lowbound}
		\begin{aligned}
			\| \opH_{\mu} u \|^2 &= \| u'' \|^2 + \| (q + 2 \mu a) u \|^2 + 2 \Re \langle u', ((q + 2 \mu a) u)' \rangle\\
			&\ge \| u'' \|^2 + \| (q + 2 \mu a) u \|^2 - 2 |\langle u', (q' + 2 \mu a') u \rangle|.
		\end{aligned}
	\end{equation}
	Applying Assumptions~\ref{asm:a.q.dwe}~\ref{itm:a.symbolclass}, \ref{itm:q.symbolclass}, with $n = 1$, there exists $C > 0$ such that for any arbitrarily small $\eps > 0$
	\begin{equation*}
		\begin{aligned}
			|\langle u', (q' + 2 \mu a') u \rangle| &\le \| u' \| \| (|q'| + 2 \mu |a'|) u \| \le C \| u' \| \left(\| (1 + 2 \mu) u \| + \| (q + 2 \mu a) u \|\right)\\
			&\le \frac{C}{2} \left(\| u' \|^2 + (1 + 2 \mu)^2 \| u \|^2 + \eps^{-1} \| u' \|^2 + \eps \| (q + 2 \mu a) u \|^2\right)\\
			&\le \frac{C}{2} \left( \eps \| u'' \|^2 + \eps \| (q + 2 \mu a) u \|^2 \right.\\
			&\qquad \left.+ \mu^2 (C_{\eps} \mu^{-2} + (2 + \mu^{-1})^2) \| u \|^2\right),
		\end{aligned}
	\end{equation*}
	with some (possibly large) constant $C_{\eps} > 0$. This shows that, for any small (but fixed) $\eps > 0$, we can find constants $C, C'_{\eps} > 0$ (independent of $\mu$) such that
	\begin{equation*}
		2 |\langle u', (q' + 2 \mu a') u \rangle| \le C \eps \left(\| u'' \|^2 + \| (q + 2 \mu a) u \|^2\right) + C'_{\eps} \mu^2 \| u \|^2, \quad \mu \to +\infty.
	\end{equation*}
	Hence by \eqref{eq:Hq.lowbound} we deduce
	\begin{equation*}
		\| \opH_{\mu} u \|^2 \ge (1 - C \eps) \left(\| u'' \|^2 + \| (q + 2 \mu a) u \|^2\right) - C'_{\eps} \mu^2 \| u \|^2, \quad \mu \to +\infty.
	\end{equation*}
	Selecting an adequately small $\eps$ and substituting in \eqref{eq:Tmu.lowbound}, we find
	\begin{equation}
		\label{eq:Tmu.graphnorm.0}
		\| u'' \|^2 + \| (q + 2 \mu a) u \|^2 \ls \| \opT(\mu) u \|^2 + \mu^2 \| u \|^2, \quad \mu \to +\infty.
	\end{equation}
	It is easy to see that
	\begin{equation*}
		\| (q + 2 \mu a) u \|^2 \ge \| q u \|^2 + 4 \mu^2 \| a u \|^2,
	\end{equation*}
	which, in combination with \eqref{eq:Tmu.graphnorm.0}, yields \eqref{eq:Tmu.graphnorm}.
	
	Clearly $q + 2\mu a + \mu^2 \ge \mu^2 > 0$ and hence $\sigma(\opT(\mu)) \subset [\mu^2, \infty)$. It follows that $\opT(\mu)$ is invertible for all $\mu > 0$. Moreover by the Rellich's criterion (see \cite[Thm.~XIII.65]{Reed4}) the set
	\begin{equation*}
		\begin{aligned}
			S &= \{\psi \in \Lt(\R): \intR |\psi(x)|^2 \dd x \le 1, \intR (q(x) + 2 \mu a(x))^2 |\psi(x)|^2 \dd x \le 1,\\
			&\hspace{1in} \intR \xi^2 |\psi(\xi)|^2 \dd \xi \le 1\}
		\end{aligned}
	\end{equation*}
	is compact and therefore, using the graph norm estimate \eqref{eq:Tmu.graphnorm}, we conclude that $\opT(\mu)$ has compact resolvent. Since $\opT(\mu)^{-1}$ is bounded and self-adjoint, we find that $\| \opT(\mu)^{-1} \| = \rad(\opT(\mu)^{-1}) \le \mu^{-2}$ which proves \eqref{eq:Tmuminus1.norm}.
	
	For $u \in \Dom(\opT(\mu)) \subset \Dom(\opH_q)$, appealing once again to \eqref{eq:Tmu.graphnorm},
	\begin{equation}
		\label{eq:Hq.Tmu.est}
		\| \opH_q u \| \le \| u'' \| + \| q u \| \ls \| \opT(\mu) u \| + \mu \| u \|, \quad \mu \to +\infty,
	\end{equation}
	therefore, letting $u := \opT(\mu)^{-1} v$ with $v \in \Lt(\R)$ such that $\| v \| \le 1$, we have by \eqref{eq:Tmuminus1.norm}
	\begin{equation*}
		\| \opH_q \opT(\mu)^{-1} v \| \ls \| v \| + \mu \| \opT(\mu)^{-1} v \| \ls 1 + \mu^{-1} \ls 1, \quad \mu \to +\infty,
	\end{equation*}
	which proves that $\| \opH_q \opT(\mu)^{-1} \| \ls 1$. Using the fact that $(\opH_q \opT(\mu)^{-1})^*$ is bounded and the property of adjoint $(\opA \opB)^* \supset \opB^* \opA^*$, if $\opA \opB$ is densely defined, we deduce that $\opT(\mu)^{-1} \opH_q$ has a bounded extension which completes the proof of \eqref{eq:Hq.Tmuminus1.norm}.
	
	For $u \in \Dom(\opT(\mu)) \subset \Dom(\opH_q)$, we have as $\mu \to +\infty$
	\begin{gather*}
		\| \opH_q^{\frac12} u \|^2 = \langle \opH_q u, u \rangle \le \| \opH_q u \| \| u \| \implies\\
		\| \opH_q^{\frac12} u \| \le \| \opH_q u \|^{\frac12} \| u \|^\frac12 \le \frac12 \left(\mu^{-1} \| \opH_q u \| + \mu \| u \|\right) \ls \mu^{-1} \| \opT(\mu) u \| + \mu \| u \|,
	\end{gather*}
	using \eqref{eq:Hq.Tmu.est} in the last step. Taking $u := \opT(\mu)^{-1} v$ as before and applying \eqref{eq:Tmuminus1.norm}, we obtain
	\begin{equation*}
		\| \opH_q^{\frac12} \opT(\mu)^{-1} v \| \ls \mu^{-1} \| v \| + \mu \| \opT(\mu)^{-1} v \| \ls \mu^{-1}, \quad \mu \to +\infty,
	\end{equation*}
	which shows that $\| \opH_q^{\frac12} \opT(\mu)^{-1} \| \ls \mu^{-1}$ and, using adjoints as above, we deduce \eqref{eq:Hq12.Tmuminus1.norm}.
	
	Finally, for any $\mu > 0$, taking $u \in \Dom(\opT(\mu)^{\frac12})$, we have
	\begin{equation*}
		\| \opT(\mu)^{\frac12} u \|^2 = \langle \opT(\mu) u, u \rangle = \langle \opH_q u, u \rangle + \langle (2 \mu a + \mu^2) u, u \rangle \ge \langle \opH_q u, u \rangle = \| \opH_q^{\frac12} u \|^2.
	\end{equation*}
	Letting $u := \opT(\mu)^{-\frac12} v$ with $v \in \Lt(\R)$ such that $\| v \| \le 1$, we deduce
	\begin{equation*}
		\| \opH_q^{\frac12} \opT(\mu)^{-\frac12} v \| \le \| v \| \le 1 \implies \| \opH_q^{\frac12} \opT(\mu)^{-\frac12} \| \le 1,
	\end{equation*}
	and, using adjoints, we obtain \eqref{eq:Hq12.Tmuminus12.norm}.
\end{proof}

\begin{proposition}
	\label{prop:Tla.graphnorm}
	Let $a$ and $q$ satisfy Assumption~\ref{asm:a.q.dwe} and let $\la$ be as in the statement of Theorem~\ref{thm:resolvent.G}. If $\opT(\la)$ is the family of operators \eqref{eq:Tla.def}-\eqref{eq:Tla.domain}, then for any $u \in \Dom(\opT(\la)) = \WttR \cap \Dom(a)$, we have
	\begin{equation}
		\label{eq:Tla.graphnorm}
		\| u'' \|^2 + \| q u \|^2 + b^2 \| a u \|^2 \ls \| \opT(\la) u \|^2 + b^4 \| u \|^2, \quad |b| \to +\infty.
	\end{equation}
\end{proposition}
\begin{proof}
	For any $u \in \Dom(\opT(\la))$, we have
	\begin{equation*}
		\begin{aligned}
			\| \opT(\la) u \|^2 &= \| u'' \|^2 + \| (q + 2 \la a +\la^2) u \|^2 + 2 \Re \langle -u'', (q + 2 \la a + \la^2) u \rangle\\
			&= \| u'' \|^2 + \| q u \|^2 + 4 |\la|^2 \| a u \|^2 + |\la|^4 \| u \|^2 + 2 \Re \langle q u, 2 \la a u \rangle\\
			&\quad + 2 \Re \langle q u, \la^2 u \rangle + 2 \Re \langle 2 \la a u, \la^2 u \rangle + 2 \Re \langle u', q' u \rangle + 2 \Re \langle u', q u' \rangle\\
			&\quad + 2 \Re \langle u', 2 \la a' u \rangle + 2 \Re \langle u', 2 \la a u' \rangle + 2 \Re \langle u', \la^2 u' \rangle.
		\end{aligned}
	\end{equation*}
	The term $\Re \langle u', q u' \rangle \ge 0$ can be dropped and, using integration by parts, we have
	\begin{equation*}
		|\Re \langle u', 2 \la a u' \rangle| = |\langle u', 2 c a u' \rangle| \le |\langle 2 c a' u, u' \rangle| + |\langle 2 c a u, u'' \rangle|.
	\end{equation*}
	Hence, taking $C_1, C'_1 > 0$ from Assumption~\ref{asm:a.q.dwe}~\ref{itm:a.symbolclass},~\ref{itm:q.symbolclass} with $n = 1$, we get
	\begin{equation*}
		\begin{aligned}
			\| \opT(\la) u \|^2 &\ge \| u'' \|^2 + \| q u \|^2 + 4 |\la|^2 \| a u \|^2 + |\la|^4 \| u \|^2 - 2 \| q u \| \| 2 c a u \|\\
			&\quad - 2 \| q u \| \| |\la|^2 u \| - 2 \| 2 |\la| a u \| \| |\la|^2 u \| - 2 \| u' \| \| C'_1 (1 + q) u \|\\
			&\quad - 2 \| u' \| \| 2 |\la| C_1 (1 + a) u \| - 2 \| u' \| \| 2 c C_1 (1 + a) u \|\\
			&\quad - 2 \| u'' \| \| 2 c a u \| - 2 |\la|^2 \| u' \|^2.
		\end{aligned}
	\end{equation*}
	Fixing a small $\eps > 0$ to be chosen below and repeatedly applying estimates such as $2 \| u \| \| v \| \le \eps \| u \|^2 + \eps^{-1} \| v \|^2$, we obtain
	\begin{equation*}
		\begin{aligned}
			\| \opT(\la) u \|^2 &\ge \| u'' \|^2 + \| q u \|^2 + 4 |\la|^2 \| a u \|^2 + |\la|^4 \| u \|^2 - \eps \| q u \|^2 - 4 c^2 \eps^{-1} \| a u \|^2\\
			&\quad -\eps \| q u \|^2 - |\la|^4 \eps^{-1} \| u \|^2 - 4 |\la|^2 \eps \| a u \|^2 - |\la|^4 \eps^{-1} \| u \|^2 - \| u' \|^2\\
			&\quad - C'^2_1 \| u \|^2 - \eps^{-1} \| u' \|^2 - C'^2_1 \eps \| q u \|^2 - \| u' \|^2 - 4 |\la|^2 C^2_1 \| u \|^2\\
			&\quad - \eps^{-1} \| u' \|^2 - 4 |\la|^2 C^2_1 \eps \| a u \|^2 - \| u' \|^2 - 4 c^2 C^2_1 \| u \|^2 - \| u' \|^2\\
			&\quad - 4 c^2 C^2_1 \| a u \|^2 - \eps \| u'' \|^2 - 4 c^2 \eps^{-1} \| a u \|^2 - 2 |\la|^2 \| u' \|^2\\
			&\ge (1 - \eps) \| u'' \|^2 - 2 |\la|^2 (1 + 2 |\la|^{-2} + \eps^{-1} |\la|^{-2}) \| u' \|^2\\
			&\quad + (1 - 2 \eps - C'^2_1 \eps) \| q u \|^2\\
			&\quad + 4 |\la|^2 (1 - 2 c^2 |\la|^{-2} \eps^{-1} - \eps - C_1^2 \eps - c^2 |\la|^{-2} C_1^2) \| a u \|^2\\
			&\quad - |\la|^4 (2 \eps^{-1} - 1 + C'^2_1 |\la|^{-4} + 4 C_1^2 |\la|^{-2} + 4 C_1^2 c^2 |\la|^{-4}) \| u \|^2.
		\end{aligned}
	\end{equation*}
	Note that for large enough $|b|$
	\begin{equation*}
		\begin{aligned}
			2 |\la|^2 (1 + 2 |\la|^{-2} + \eps^{-1} |\la|^{-2}) \| u' \|^2 &\le 2 |\la|^2 (1 + 2 |\la|^{-2} + \eps^{-1} |\la|^{-2}) \| u \| \| u'' \|\\
			&\le \eps \| u'' \|^2 + |\la|^4 C_\eps \| u \|^2
 		\end{aligned}
	\end{equation*}
	for some $C_\eps > 0$ independent of $\la$. Hence as $|b| \to +\infty$
	\begin{equation*}
		\begin{aligned}
			\| \opT(\la) u \|^2 &\ge (1 - 2 \eps) \| u'' \|^2 + (1 - 2 \eps - C'^2_1 \eps) \| q u \|^2\\
			&\quad + 4 |\la|^2 (1 - 2 c^2 |\la|^{-2} \eps^{-1} - \eps - C_1^2 \eps - c^2 |\la|^{-2} C_1^2) \| a u \|^2\\
			&\quad - |\la|^4 (2 \eps^{-1} + C_\eps - 1 + C'^2_1 |\la|^{-4} + 4 C_1^2 |\la|^{-2} + 4 C_1^2 c^2 |\la|^{-4}) \| u \|^2.
		\end{aligned}
	\end{equation*}
	Furthermore, since $c \in K$, a bounded subset of $\R$, we have $c \ls_K 1$ and $|\la| \approx |b|$ as $|b| \to +\infty$. Therefore, choosing an adequately small $\eps$, we obtain (with implicit constant independent of $\la$, see also Remark~\ref{rmk:b.asymptotic})
	\begin{equation*}
		\begin{aligned}
			\| \opT(\la) u \|^2 \gs \| u'' \|^2 + \| q u \|^2 + b^2 \| a u \|^2 - b^4 \| u \|^2, \quad |b| \to +\infty,
		\end{aligned}
	\end{equation*}
	which proves \eqref{eq:Tla.graphnorm}.
\end{proof}

Our main result in this section is an asymptotic estimate for $\| \opT(\la)^{-1} \|$ along vertical strips inside $\overline \C_{-}$.
\begin{theorem}
	\label{thm:resolvent.Tla}
	Let $a$ and $q$ satisfy Assumption~\ref{asm:a.q.dwe} and let $\opT(\la)$ be the family of operators \eqref{eq:Tla.def}-\eqref{eq:Tla.domain} for $\la := -c + i b$ as in the statement of Theorem~\ref{thm:resolvent.G}. Then
	\begin{equation}
		\label{eq:resnorm.Tla}
		\| \opT(\la)^{-1} \| = \| (\opA - c)^{-1} \| (2 |b|)^{-1} (1 + \BigO_K(|b|^{-1})), \quad |b| \to +\infty,
	\end{equation}
	with $\opA$ as in \eqref{eq:Airy.def}.
\end{theorem}

\begin{remark}
	The conditions on $a$ in Assumption~\ref{asm:a.q.dwe} ensure that the generalised Airy operator $\opA$ in \eqref{eq:resnorm.Tla} satisfies the properties listed in Sub-section~\ref{ssec:Airy.prelim} (see \cite[Prop.~A.1, Prop.~A.2]{ArSi-resolvent-2022} for details).
\end{remark}

\begin{remark}
	\label{rmk:Tla.res.norm.uniform.bound}
	Since $\sigma(\opA) = \emptyset$, it follows that there exists $M_K > 0$ such that $\| (\opA - c)^{-1} \| \le M_K$ for all $c \in K$.
\end{remark}

Before proving Theorem~\ref{thm:resolvent.Tla}, we present some immediate consequences.
\begin{corollary}
	With $a$, $q$, $\la$ and $\opT(\la)$ as in Proposition~\ref{prop:Tla.graphnorm} and $\opH_q$ as in \eqref{eq:Hq.def}, then for $|b| \to +\infty$
	\begin{gather}
		\label{eq:Hq.Tlaminus1.norm}\| \opH_q \opT(\la)^{-1} \| + \| \opT(\la)^{-1} \opH_q \| \ls_K |b|,\\
		\label{eq:Hq12.Tlaminus1.norm}\| \opH^{\frac12}_q \opT(\la)^{-1} \| + \| \opT(\la)^{-1} \opH^{\frac12}_q \| \ls_K 1.
	\end{gather}
\end{corollary}
\begin{proof}
	Let $u \in \Dom(\opT(\la)) \subset \Dom(\opH_q)$, then by \eqref{eq:Tla.graphnorm}
	\begin{equation*}
		\| \opH_q u \| \le \| u'' \| + \| q u \| \ls \| \opT(\la) u \| + b^2 \| u \|, \quad |b| \to +\infty.
	\end{equation*}
	Taking $u := \opT(\la)^{-1} v$, with $v \in \Lt(\R)$, $\| v \| \le 1$, we have
	\begin{equation*}
		\| \opH_q \opT(\la)^{-1} v \| \ls \| v \| + b^2 \| \opT(\la)^{-1} v \| \ls_K |b|, \quad |b| \to +\infty,
	\end{equation*}
	where we have used $\| \opT(\la)^{-1} \| \ls_K |b|^{-1}$ (see \eqref{eq:resnorm.Tla} and Remark~\ref{rmk:Tla.res.norm.uniform.bound}). This proves that $\| \opH_q \opT(\la)^{-1} \| \ls_K |b|$. A by now familiar use of adjoints and the fact that $\opT(\la)^* = \opT(\overline{\la})$ (see our observations in Sub-section~\ref{ssec:dwe.prelim}) yield \eqref{eq:Hq.Tlaminus1.norm}.
	
	Let $u \in \Dom(\opT(\la)) \subset \Dom(\opH_q)$, then applying once more \eqref{eq:Tla.graphnorm} we derive
	\begin{equation*}
		\begin{aligned}
			\| \opH_q^{\frac12} u \|^2 &= \langle \opH_q u, u \rangle \le \| \opH_q u \| \| u \| \le (\| u'' \| + \| q u \|) \| u \|\\
			&\ls (\| \opT(\la) u \| + b^2 \| u \|) \| u \|
		\end{aligned}
	\end{equation*}
	as $|b| \to +\infty$. It follows
	\begin{equation*}
		\| \opH_q^{\frac12} u \| \ls \| \opT(\la) u \|^{\frac12} \| u \|^{\frac12} + |b| \| u \| \ls \| \opT(\la) u \| + |b| \| u \|, \quad |b| \to +\infty,
	\end{equation*}
	and therefore arguing as above and applying \eqref{eq:resnorm.Tla}
	\begin{equation*}
		\| \opH_q^{\frac12} \opT(\la)^{-1} v \| \ls \| v \| + |b| \| \opT(\la)^{-1} v \| \ls_K 1, \quad |b| \to +\infty,
	\end{equation*}
	which, repeating previous arguments, proves \eqref{eq:Hq12.Tlaminus1.norm}.
\end{proof}

\subsection{Proof of Theorem~\ref{thm:resolvent.Tla}}
\label{ssec:proof.resolvent.Tla}
The strategy of the proof follows the template laid out in \cite[Sec.~4]{ArSi-resolvent-2022} to analyse the norm of the resolvent in the real axis for Schr\"{o}dinger operators with complex potentials. We firstly transform the problem to Fourier space. We then study the resolvent norm of the transformed operator in four steps: find an estimate away from the (asymptotic) zeroes of its potential function (i.e. the non-pseudo-differential term), find a local estimate near the zeroes, find a lower bound for the norm and, finally, combine the previous results to prove the theorem.

To this end, let us introduce the operators in $L^2(\R)$
\begin{equation}
	\label{eq:Tlahat.def}
	\begin{aligned}
		\widehat \opT(\la) &:=  \sF \opT(\la) \sF^{-1}, &\Dom(\widehat \opT(\la)) &:= \{u \in \Lt(\R): \quad \check{u} \in \Dom(\opT(\la))\},\\
		\widehat{a} &:= \sF a \sF^{-1}, &\Dom(\widehat{a}) &:= \{u \in \Lt(\R): \quad \check{u} \in \Dom(a)\},\\
		\widehat{q} &:= \sF q \sF^{-1}, &\Dom(\widehat{q}) &:= \{u \in \Lt(\R): \quad \check{u} \in \Dom(q)\}.
	\end{aligned}
\end{equation}
Notice that $\widehat \opT(\la) = \widehat{q} + 2 \la \widehat{a} + \xi^2 + \la^2$ and, with $\la$ as in the statement of Theorem~\ref{thm:resolvent.Tla}, we have $\xi^2 + \la^2 = \xi^2 - b^2 (1 + \BigO(|b|^{-1}))$ as $|b| \to +\infty$.

Furthermore, since $\opT(\overline{\la}) = \opT(\la)^*$ (refer to our remarks in Sub-section~\ref{ssec:dwe.prelim}), it is enough to prove the theorem for $b \to +\infty$ and we will therefore assume $b > 0$ in the rest of the section. Let
\begin{equation}
	\label{eq:deltab.def}
	\Omega'_{b,\pm} := ( \pm \xi_b - \delta_b,  \pm \xi_b + \delta_b ), \quad \xi_b := b, \quad \delta_b := \delta  \xi_b, \quad 0 < \delta < \frac14,
\end{equation}
where the parameter $\delta = \delta(K)$ will be specified in Proposition~\ref{prop:local.Tla}.

\subsubsection{Step 1: estimate outside the neighbourhoods of $\pm \xi_b$}

\begin{proposition}
	\label{prop:away.Tla}
	Let $\Omega'_{b,\pm}$ be defined by~\eqref{eq:deltab.def}, let the assumptions of Theorem~\ref{thm:resolvent.Tla} hold and let $\widehat \opT(\la)$ be as in \eqref{eq:Tlahat.def}. Then as $b \to +\infty$
	\begin{equation}
		\label{eq:away.est}
		b^2 \ls_\delta \inf \left\{ \frac{\| \widehat \opT(\la)  u \|}{\|u\|} : \; 0 \neq u \in \Dom(\widehat \opT(\la)), \; \supp u \cap (\Omega'_{b,+} \cup \Omega'_{b,-}) = \emptyset \right\}.
	\end{equation}
\end{proposition}
\begin{proof}
	Let $0 \ne u \in \Dom(\widehat \opT(\la))$ with $\supp u \cap (\Omega'_{b,+} \cup \Omega'_{b,-}) = \emptyset$ and consider
	\begin{equation}
		\label{eq:Tlahat.away.est.0}
		\begin{aligned}
			\| \widehat \opT(\la) u \|^2 &= \| (\widehat{q} + 2 \la \widehat{a}) u \|^2 + \| (\xi^2 + \la^2)  u \|^2 + 2 \Re \langle (\widehat{q} + 2 \la \widehat{a}) u, (\xi^2 + \la^2) u \rangle\\
			&\ge \| (\widehat{q} + 2 \la \widehat{a}) u \|^2 + \| (\xi^2 + \la^2)  u \|^2 - 2 |\Re \langle (\widehat{q} + 2 \la \widehat{a}) u, (\xi^2 + \la^2) u \rangle|.
		\end{aligned}
	\end{equation}

	Note that
	\begin{equation*}
		\begin{aligned}
			\| (\widehat{q} + 2 \la \widehat{a}) u \|^2 &= \| \widehat{q} u \|^2 + 4 |\la|^2 \| \widehat{a} u \|^2 + 2 \Re \langle \widehat{q} u, 2 \la \widehat{a} u \rangle\\
			&= \| \widehat{q} u \|^2 + 4 |\la|^2 \| \widehat{a} u \|^2 - 2 \langle \widehat{q} u, 2 c \widehat{a} u \rangle\\
			&\ge \| \widehat{q} u \|^2 + 4 |\la|^2 \| \widehat{a} u \|^2 - 2 \| \widehat{q} u \| \| 2 c \widehat{a} u \|\\
			&\ge \frac12 \| \widehat{q} u \|^2 + 4 b^2 (|\la|^2 b^{-2} - 2 c^2 b^{-2}) \| \widehat{a} u \|^2.
		\end{aligned}
	\end{equation*}
	Hence we can find $C''_1 > 0$ such that
	\begin{equation}
		\label{eq:Tlahat.away.est.1}
		\| \widehat{q} u + 2 \la \widehat{a} u \|^2 \ge C''_1 \left(\| \widehat{q} u \|^2 + b^2 \| \widehat{a} u \|^2\right), \quad b \to +\infty.
	\end{equation}
	
	Next we estimate the third term in the right-hand side of \eqref{eq:Tlahat.away.est.0}. For any arbitrarily small $\eps > 0$
	\begin{equation*}
		\begin{aligned}
			2 |\Re \langle \widehat{q} u, (\xi^2 + \la^2) u \rangle| &\le 2 \| \widehat{q} u \| \| (\xi^2 + \la^2) u\| \le \eps^{-1} \| \widehat{q} u \|^2 + \eps \| (\xi^2 + \la^2) u\|^2,\\
			2 |\Re \langle 2 \la \widehat{a} u, (\xi^2 + \la^2) u \rangle| &\le 4 |\Re \la \langle \widehat{a} u, \xi^2 u \rangle| + 4 c |\la|^2 \langle \widehat{a} u, u \rangle\\
			&\le 4 c |\Re \langle \widehat{a} u, \xi^2 u \rangle| + 4 b |\Im \langle \widehat{a} u, \xi^2 u \rangle| + 4 c |\la|^2 \langle \widehat{a} u, u \rangle\\
			&\le 2 \| \widehat{a} u \| \| 2 c \xi^2 u \| + 4 b |\Im \langle a' \check{u}, \check{u}' \rangle| + 2 |\la|^2 \| \widehat{a} u \| \| 2 c u \|\\
			&\le \eps b^2 \| \widehat{a} u \|^2 + \eps^{-1} b^{-2} 4 c^2 \| \xi^2 u \|^2 + 4 b |\Im \langle a' \check{u}, \check{u}' \rangle|\\
			&\quad + |\la|^2 (\eps \| \widehat{a} u \| + \eps^{-1} 4 c^2 \| u \|^2).
		\end{aligned}
	\end{equation*}
	Applying Assumption~\ref{asm:a.q.dwe}~\ref{itm:a.symbolclass} with $n = 1$, we obtain
	\begin{equation*}
		\begin{aligned}
			4 b |\Im \langle a' \check{u}, \check{u}' \rangle| &\le 4 b C_1 \| (1 + \widehat{a}) u \| \| \xi u \| \le 2 b C_1 (2 \| u \| \| \xi u \| + 2 \| \widehat{a} u \| \| \xi u \|)\\
			&\le 2 b C_1 \left(\| u \|^2 + \| \widehat{a} u \|^2 + 2 \| \xi u \|^2\right)\\
			&\le 2 b C_1 \left(\| u \|^2 + \| \widehat{a} u \|^2 + 2 \| \xi^2 u \| \| u \|\right)\\
			&\le 2 b C_1 \left(\| u \|^2 + \| \widehat{a} u \|^2 + b^{-2} \| \xi^2 u \|^2 + b^2 \| u \|^2\right).
		\end{aligned}
	\end{equation*}
	Hence there exists $C''_2 > 0$ such that as $b \to +\infty$
	\begin{equation*}
		\begin{aligned}
			2 |\Re \langle 2 \la \widehat{a} u, (\xi^2 + \la^2) u \rangle| &\le C''_2 (\eps b^2 \| \widehat{a} u \|^2 + (b^{-1} + \eps^{-1} b^{-2} c^2) \| \xi^2 u \|^2\\
			&\hspace{0.4in} + b^4 (b^{-1} + \eps^{-1} b^{-2} c^2) \| u \|^2),
		\end{aligned}
	\end{equation*}
	and consequently as $b \to +\infty$
	\begin{equation}
		\label{eq:Tlahat.away.est.2}
		\begin{aligned}
			2 |\Re \langle (\widehat{q} + 2 \la \widehat{a}) u, (\xi^2 + \la^2) u \rangle| &\le \eps^{-1} \| \widehat{q} u \|^2 + \eps \| (\xi^2 + \la^2) u\|^2\\
			&\quad + C''_2 (\eps b^2 \| \widehat{a} u \|^2 + (b^{-1} + \eps^{-1} b^{-2} c^2) \| \xi^2 u \|^2\\
			&\hspace{0.5in} + b^4 (b^{-1} + \eps^{-1} b^{-2} c^2) \| u \|^2).
		\end{aligned}
	\end{equation}

	Substituting \eqref{eq:Tlahat.away.est.1} and \eqref{eq:Tlahat.away.est.2} in \eqref{eq:Tlahat.away.est.0}, we have as $b \to +\infty$
	\begin{equation*}
		\begin{aligned}
			\| \widehat \opT(\la) u \|^2 &\ge (C''_1 - \eps^{-1}) \| \widehat{q} u \|^2 + b^2 (C''_1 - C''_2 \eps) \| \widehat{a} u \|^2 + (1 - \eps) \| (\xi^2 + \la^2) u\|^2\\
			&\quad - C''_2 (b^{-1} + \eps^{-1} b^{-2} c^2) (\| \xi^2 u \|^2 + b^4 \| u \|^2)
		\end{aligned}
	\end{equation*}
	and therefore, choosing a small enough $\eps$, we can find $C''_3, C''_4 > 0$ such that
	\begin{equation}
		\label{eq:Tlahat.away.est.3}
		\begin{aligned}
			\| \widehat \opT(\la) u \|^2 &\ge C''_3 (b^2 \| \widehat{a} u \|^2 + \| (\xi^2 + \la^2) u\|^2)\\
			&\quad - C''_4 (\| \widehat{q} u \|^2 + (b^{-1} + b^{-2} c^2) (\| \xi^2 u \|^2 + b^4 \| u \|^2)), \quad b \to +\infty.
		\end{aligned}
	\end{equation}

	Finally, with $\xi^2 + \la^2 = \xi^2 - b^2 + c^2 - 2 i c b$, we consider the term
	\begin{equation*}
		\begin{aligned}
			\| (\xi^2 + \la^2) u\|^2 &= \| (\xi^2 - b^2) u\|^2 + c^2 (c^2 + 4 b^2) \| u \|^2 + 2 c^2 \langle (\xi^2 - b^2) u, u \rangle\\
			&\ge \| (\xi^2 - b^2) u\|^2 + 4 c^2 b^2 \| u \|^2 - 2 \| (\sqrt{2})^{-1} (\xi^2 - b^2) u \| \| \sqrt{2} c^2 u \| \\
			&\ge \frac12 \| (\xi^2 - b^2) u\|^2 + 2 c^2 b^2 (2 - c^2 b^{-2}) \| u \|^2
		\end{aligned}
	\end{equation*}
	and hence
	\begin{equation*}
		\| (\xi^2 + \la^2) u\|^2 \ge \frac12 \| (\xi^2 - b^2) u\|^2, \quad b \to +\infty.
	\end{equation*}
	Furthermore, there exists $C_\delta' > 0$, depending on $\delta$, such that for any $\xi \in \supp u$
	\begin{gather*}
		|\xi^2 - b^2| = |\xi + \xi_b| |\xi - \xi_b| \ge \delta_b^2 = \delta^2 b^2,\\
		|\xi| \le |\xi \pm \xi_b| + \xi_b \le (1 + 1/\delta) |\xi \pm \xi_b| \implies |\xi^2 - b^2| \ge C_\delta' \xi^2.
	\end{gather*}
	Consequently there exists $C''_{5,\delta} > 0$ such that for $b \to +\infty$
	\begin{equation}
		\label{eq:Tlahat.away.est.4}
		\| (\xi^2 + \la^2) u\|^2 \ge \frac14 \| (\xi^2 - b^2) u\|^2 + \frac14 \| (\xi^2 - b^2) u\|^2 \ge C''_{5,\delta} (\| \xi^2 u \|^2 + b^4 \| u \|^2).
	\end{equation}

	Replacing \eqref{eq:Tlahat.away.est.4} in \eqref{eq:Tlahat.away.est.3}, we deduce that there exists $C''_{6,\delta} > 0$ such that
	\begin{equation*}
		\begin{aligned}
			\| \widehat \opT(\la) u \|^2 &\ge C''_{6,\delta} (b^2 \| \widehat{a} u \|^2 + \| \xi^2 u \|^2 + b^4 \| u \|^2)\\
			&\quad - C''_4 (\| \widehat{q} u \|^2 + (b^{-1} + b^{-2} c^2) (\| \xi^2 u \|^2 + b^4 \| u \|^2)), \quad b \to +\infty,
		\end{aligned}
	\end{equation*}
	and, noting that $\| \widehat{q} u \|^2 \ls \| \widehat{a} u \|^2 + \| u \|^2$ (see Assumption~\ref{asm:a.q.dwe}~\ref{itm:a.gt.q}) and $b^{-2} c^2 \to 0$ as $b \to +\infty$ for $c \in K$, we conclude that there exists $C'_{\delta} > 0$ such that
	\begin{equation}
		\label{eq:Tlahat.away.est.5}
		\| \widehat \opT(\la) u \|^2 \ge C'_{\delta} (b^2 \| \widehat{a} u \|^2 + \| \xi^2 u \|^2 + b^4 \| u \|^2), \quad b \to +\infty,
	\end{equation}
	which proves the claim.
\end{proof}

\subsubsection{Step 2: estimate near $\pm \xi_b$}

\begin{proposition}
	\label{prop:local.Tla}
	Define 
	\begin{equation}
		\label{eq:Omega.def}
		\Omega_{b,\pm} := \left( \pm \xi_b - 2 \delta_b, \pm \xi_b + 2 \delta_b \right), 
	\end{equation}
	with $\xi_b$, $\delta_b$ as in \eqref{eq:deltab.def}. Let the assumptions of Theorem~\ref{thm:resolvent.Tla} hold and let $\widehat \opT(\la)$ and $\opA$ be as in \eqref{eq:Tlahat.def} and \eqref{eq:Airy.def}, respectively. Then as $b \to + \infty$
	\begin{equation}
		\label{eq:local.est}
		\begin{aligned}
			&\| (\opA - c)^{-1} \|^{-1} 2 b (1 - \BigO_K(b^{-1}))\\
			&\qquad \qquad \le \inf \left\{ \frac{\| \widehat \opT(\la) u \|}{\|u\|} : \; 0 \neq u \in \Dom(\widehat \opT(\la)), \, \supp u \subset \Omega_{b,\pm} \right\}.
		\end{aligned}
	\end{equation}
\end{proposition}
\begin{proof}
	We shall derive estimate \eqref{eq:local.est} for $u$ such that $\supp u \subset \Omega_{b,+}$. The procedure when $\supp u \subset \Omega_{b,-}$ is similar (see our remarks at the end of the proof).

	Writing $\xi^2 - b^2 = 2 \xi_b (\xi - \xi_b) + (\xi - \xi_b)^2$, we introduce
	\begin{equation*}
		\widetilde V_b(\xi) := c^2 - 2 i c \xi_b + 2  \xi_b (\xi - \xi_b) + (\xi - \xi_b)^2 \chi_{\Omega_{b,+}}(\xi), \quad \xi \in \R.
	\end{equation*}
	With $\widehat{q}$ and $\widehat{a}$ as in \eqref{eq:Tlahat.def}, let us define the following operator in $\Lt(\R)$
	\begin{equation*}
		\widetilde \opT(\la) = \widehat{q} + 2 \la \widehat{a} + \widetilde V_b(\xi), \quad \Dom(\widetilde \opT(\la)) = \left\{ u \in \Lt(\R) \, : \, \check{u} \in W^{1,2}(\R) \cap \Dom(a) \right\}.
	\end{equation*}

	We define a translation on $L^2(\R)$ by
	\begin{equation*}
		(\opU_b u)(\xi) := u(\xi + \xi_b), \quad \xi \in \R.
	\end{equation*}
	Then, setting $\Omega_b := (-2 \delta_b, 2 \delta_b)$, we have
	\begin{equation}
		\label{eq:U.That.Uminus1}
		\begin{aligned}
			\frac1{2 \la}\opU_b \widetilde \opT(\la) \opU_b^{-1} &= \widehat{a} + \frac1{2 \la} \widehat{q} + \frac{c^2 - 2 i c \xi_b}{2 \la} + \frac{\xi_b}{\la} \xi + \frac1{2 \la} \xi^2 \chi_{\Omega_b}\\
			&= \widehat{a} - i \xi - c + \frac1{2 \la} \widehat{q} + \frac{c^2 - 2 i c \xi_b + 2 c \la}{2 \la} + \frac{\xi_b + i \la}{\la} \xi + \frac1{2 \la} \xi^2 \chi_{\Omega_b}\\
			&= \widehat{a} - i \xi - c + \frac1{2 \la} \widehat{q} - \frac{c^2}{2 \la} - i \frac{c}{\la} \xi + \frac1{2 \la} \xi^2 \chi_{\Omega_b}\\
			&= \widehat{a} - i \xi - c + \frac1{2 \la} \widehat{q} + \widehat R_b (\xi),
		\end{aligned}
	\end{equation}
	with
	\begin{equation}
		\label{eq:Rbhat.def}
		\widehat R_b (\xi) := -\frac{c^2}{2 \la} - i \frac{c}{\la} \xi + \frac1{2 \la} \xi^2 \chi_{\Omega_b}(\xi), \quad \xi \in \R.
	\end{equation}
	From \eqref{eq:deltab.def}, we have $\delta_b = \delta \xi_b = \delta b$ and, since $b \le |\la|$, we find
	\begin{equation}
		\label{eq:Rbhat.est}
		\begin{aligned}
			&\| \xi^{-1} (\widehat R_b + \frac{c^2}{2 \la} + i \frac{c}{\la} \xi) \|_\infty = \frac{\| \xi \chi_{\Omega_b} \|_\infty}{2 |\la|}  \le \delta,\\
			&\| \xi^{-2} (\widehat R_b + \frac{c^2}{2 \la} + i \frac{c}{\la} \xi) \|_\infty \le \frac1{2 |\la|} \le \frac1{2 b}.
		\end{aligned}
	\end{equation}

	Using \eqref{eq:U.That.Uminus1} and letting
	\begin{align}
		\label{eq:Sinfhat.def}\widehat \opS_{\infty} &:= \sF \opA \sF^{-1} = \widehat a - i \xi, \quad \Dom(\widehat \opS_{\infty}) = \left\{u \in \Lt(\R) : \, \check{u} \in \Dom(\opA)\right\},\\
		\label{eq:Sbhat.def}\widehat \opS_b &:= \frac1{2 \la}\opU_b \widetilde \opT(\la) \opU_b^{-1} + c = \widehat \opS_{\infty} + \frac1{2 \la} \widehat q + \widehat R_b, \quad \Dom(\widehat \opS_b) = \Dom(\widehat \opS_{\infty}),
	\end{align}
	our next aim is to show that the operator $\widehat \opS_b - c$ converges to $\widehat \opS_\infty - c$ in the norm resolvent sense as $b \to +\infty$.
	
	The spectrum of $\opA$, and hence that of $\widehat \opS_{\infty}$, is empty (refer to Sub-section~\ref{ssec:Airy.prelim}) and therefore $\| (\widehat \opS_{\infty} - c)^{-1} \| \ls_K 1$ for all $c \in K$ (see Remark~\ref{rmk:Tla.res.norm.uniform.bound}). Moreover, using standard arguments, the graph-norm inequalities \eqref{eq:Abeta.graphnorm.realline.def} can be extended to show
	\begin{equation*}
		\| (\opA - c) u \|^2 + \langle c \rangle^2 \| u \|^2 \gs \|  u' \|^2 + \left\| a u \right\|^2 + \| u \|^2, \quad u \in \Dom(\opA),
	\end{equation*}	
	(where the implicit constant is independent of $c$), which on Fourier space reads
	\begin{equation}
		\label{eq:Sinfhat.graphnorm.est}
		\| (\widehat \opS_{\infty} - c) u \|^2 + \langle c \rangle^2 \| u \|^2 \gs \|  \xi u \|^2 + \left\| \widehat{a} u \right\|^2 + \| u \|^2, \quad u \in \Dom(\widehat \opS_{\infty}).
	\end{equation}	
	From \eqref{eq:Sinfhat.graphnorm.est}, reasoning as in the proof of \eqref{eq:Hq.Tmuminus1.norm}, we deduce
	\begin{equation}
		\label{eq:Sinfhatinv.graphnorm.est.1}
		\begin{aligned}
			&\| \xi (\widehat \opS_{\infty} - c)^{-1} \| + \| (\widehat \opS_{\infty} - c)^{-1} \xi \| + \| \widehat{a} (\widehat \opS_{\infty} - c)^{-1} \| + \| (\widehat \opS_{\infty} - c)^{-1} \widehat{a} \|\\
			&\hspace{2in} \ls 1 + \langle c \rangle \| (\widehat \opS_{\infty} - c)^{-1} \|.
		\end{aligned}
	\end{equation}
	Furthermore, by Assumption~\ref{asm:a.q.dwe}~\ref{itm:a.gt.q}, we have $\| \widehat{q} u \| \ls \| \widehat{a} u \| + \| u \|$ and hence
	\begin{equation}
		\label{eq:Sinfhatinv.graphnorm.est.2}
		\| \widehat{q} (\widehat \opS_{\infty} - c)^{-1} \| + \| (\widehat \opS_{\infty} - c)^{-1} \widehat{q} \| \ls 1 + \langle c \rangle \| (\widehat \opS_{\infty} - c)^{-1} \|.
	\end{equation}

	Let us write
	\begin{equation}
		\label{eq:Sbhat.Sbinf}
		\widehat \opS_b - c = \left(\opI + \frac1{2 \la} \widehat q (\widehat \opS_{\infty} - c)^{-1} + \widehat R_b (\widehat \opS_{\infty} - c)^{-1}\right) (\widehat \opS_{\infty} - c).
	\end{equation}
	Note that, by \eqref{eq:Rbhat.def}, \eqref{eq:Rbhat.est}, \eqref{eq:Sinfhatinv.graphnorm.est.1} and \eqref{eq:Sinfhatinv.graphnorm.est.2}, we have
	\begin{equation*}
		\begin{aligned}
			\| \frac1{2 \la} \widehat q (\widehat \opS_{\infty} - c)^{-1} \| &\ls \frac{1 + \langle c \rangle \| (\widehat \opS_{\infty} - c)^{-1} \|}{b},\\
			\| \widehat \opR_b (\widehat \opS_{\infty} - c)^{-1} \| &\le \frac{c^2}{2 |\la|} \| (\widehat \opS_{\infty} - c)^{-1} \| + \frac{c}{|\la|}\| \xi (\widehat \opS_{\infty} - c)^{-1} \|\\
			&\quad + \| \xi^{-1} (\widehat R_b + \frac{c^2}{2 \la} + i \frac{c}{\la} \xi) \|_\infty \| \xi (\widehat \opS_{\infty} - c)^{-1} \|\\
			&\ls c^2 b^{-1} \| (\widehat \opS_{\infty} - c)^{-1} \| + (c b^{-1} + \delta) (1 + \langle c \rangle \| (\widehat \opS_{\infty} - c)^{-1} \|),
		\end{aligned}
	\end{equation*}
	and it therefore follows that there exists a large enough $b_0(K) > 0$ and a sufficiently small $\delta(K) > 0$ (independent of $b$) such that the operator $\opI + (2 \la)^{-1} \widehat q (\widehat \opS_{\infty} - c)^{-1} + \widehat R_b (\widehat \opS_{\infty} - c)^{-1}$ is bounded (with $\| \opI + (2 \la)^{-1} \widehat q (\widehat \opS_{\infty} - c)^{-1} + \widehat R_b (\widehat \opS_{\infty} - c)^{-1} \| \approx 1$) and invertible for $b \ge b_0$. Hence using \eqref{eq:Sinfhat.graphnorm.est} and \eqref{eq:Sbhat.Sbinf} we deduce
	\begin{equation}
		\label{eq:Sbhat.graphnorm.est}
		\| (\widehat \opS_b - c) u \|^2 + \langle c \rangle^2 \| u \|^2 \gs \|  \xi u \|^2 + \left\| \widehat{a} u \right\|^2 + \| u \|^2, \quad u \in \Dom(\widehat \opS_b), \quad b \to +\infty.
	\end{equation}	
	Moreover by \eqref{eq:Sbhat.Sbinf} we find that $\widehat \opS_b - c$ is invertible and
	\begin{equation}
		\label{eq:Sbhat.inverse}
		(\widehat \opS_b - c)^{-1} = (\widehat \opS_{\infty} - c)^{-1} \left(\opI + \frac1{2 \la} \widehat q (\widehat \opS_{\infty} - c)^{-1} + \widehat R_b (\widehat \opS_{\infty} - c)^{-1}\right)^{-1},
	\end{equation}
	for $b \to +\infty$. Therefore by \eqref{eq:Sinfhatinv.graphnorm.est.1} and \eqref{eq:Sbhat.inverse}
	\begin{equation}
		\label{eq:Sbhatinv.graphnorm.est}
		\begin{aligned}
			&\| (\widehat \opS_b - c)^{-1} \| \approx \| (\widehat \opS_{\infty} - c)^{-1} \|,\\
			&\| \xi (\widehat \opS_b - c)^{-1} \| + \| (\widehat \opS_b - c)^{-1} \xi \| + \| \widehat{a} (\widehat \opS_b - c)^{-1} \| + \| (\widehat \opS_b - c)^{-1} \widehat{a} \|\\
			&\hspace{2.5in}\ls 1 + \langle c \rangle \| (\widehat \opS_{\infty} - c)^{-1} \|,
		\end{aligned}
	\end{equation}
	for $b \to +\infty$.
	
	Applying the second resolvent identity, we have
	\begin{equation}
		\label{eq:Sbhat.Sinfhat.est}
		\begin{aligned}
			\| (\widehat \opS_b - c)^{-1} - (\widehat \opS_{\infty} - c)^{-1} \| &\le \frac1{2 |\la|} \| (\widehat \opS_b - c)^{-1} \widehat{q} (\widehat \opS_{\infty} - c)^{-1} \|\\
			&\quad + \| (\widehat \opS_b - c)^{-1} (\frac{c^2}{2 \la} + i \frac{c}{\la} \xi) (\widehat \opS_{\infty} - c)^{-1} \|\\
			&\quad + \| (\widehat \opS_b - c)^{-1}  \xi  \xi^{-2} (\widehat R_b + \frac{c^2}{2 \la} + i \frac{c}{\la} \xi) \xi  (\widehat \opS_{\infty} - c)^{-1} \|\\
			&\ls b^{-1} \| (\widehat \opS_{\infty} - c)^{-1} \| (1 + \langle c \rangle \| (\widehat \opS_{\infty} - c)^{-1} \|)\\
			&\quad + c^2 b^{-1} \| (\widehat \opS_{\infty} - c)^{-1} \|^2\\
			&\quad + c b^{-1} \| (\widehat \opS_{\infty} - c)^{-1} \| (1 + \langle c \rangle \| (\widehat \opS_{\infty} - c)^{-1} \|)\\
			&\quad + b^{-1} (1 + \langle c \rangle \| (\widehat \opS_{\infty} - c)^{-1} \|)^2\\
			&\ls_K b^{-1} \| (\widehat \opS_{\infty} - c)^{-1} \|,
		\end{aligned}
	\end{equation}
	as $b \to +\infty$, where we have used \eqref{eq:Rbhat.est}, \eqref{eq:Sinfhatinv.graphnorm.est.1}, \eqref{eq:Sinfhatinv.graphnorm.est.2}, \eqref{eq:Sbhatinv.graphnorm.est} and the fact that the resolvent of $\widehat \opS_{\infty}$ is bounded above and below on K. Thus
	\begin{equation*}
		\| (\widehat \opS_b - c)^{-1} \| = \| (\widehat \opS_{\infty} - c)^{-1} \|  (1 + \BigO_K(b^{-1})), \quad b \to +\infty.
	\end{equation*}

	Since $\widehat \opS_b - c = (2 \la)^{-1} \opU_b \widetilde \opT(\la) \opU_b^{-1}$ and moreover $\| \widetilde \opT(\la)  u \| = \| \widehat \opT(\la) u \|$ for $0 \ne u \in \Dom(\widehat \opT(\la))$ such that $\supp u \subset \Omega_{b,+}$, we arrive at
	\begin{equation*}
		2 |\la| \| u \| = 2 |\la| \| \widetilde \opT(\la)^{-1} \widetilde \opT(\la) u \| \le \| (\widehat \opS_{\infty} - c)^{-1} \|  (1 + \BigO_K(b^{-1})) \| \widehat \opT(\la)  u \|, \quad b \to +\infty,
	\end{equation*}
	as required.
	
	For the case $\supp u \subset \Omega_{b,-}$, we repeat the above arguments but defining instead $\widetilde V_b(\xi) := c^2 - 2 i c \xi_b - 2  \xi_b (\xi + \xi_b) + (\xi + \xi_b)^2 \chi_{\Omega_{b,-}}(\xi)$, $(\opU_b u)(\xi) := u(\xi - \xi_b)$, $\widehat R_b (\xi) := -(2 \la)^{-1} c^2 + i (\la)^{-1} c \xi + (2 \la)^{-1} \xi^2 \chi_{\Omega_b}(\xi)$ and $\widehat \opS_{\infty} = \sF \opA^* \sF^{-1} = \widehat{a} + i \xi$.
\end{proof}

\subsubsection{Step 3: lower estimate}

\begin{proposition}
	\label{prop:lbound.Tla}
	Let the assumptions of Theorem~\ref{thm:resolvent.Tla} hold and let $\widehat \opT(\la)$ and $\opA$ be as in \eqref{eq:Tlahat.def} and \eqref{eq:Airy.def}, respectively. Then there exist functions $0 \neq u_b \in \Dom(\widehat \opT(\la))$ such that
	\begin{equation*}
		\| \widehat \opT(\la) u_b \| = \| (\opA - c)^{-1} \|^{-1} 2 b (1 + \BigO_K(b^{-1})) \| u_b \|, \quad b \to + \infty.
	\end{equation*}
\end{proposition}
\begin{proof}
	We retain the notation introduced in the proof of Proposition~\ref{prop:local.Tla}; in particular, $\widehat \opS_{\infty}$ and $\widehat \opS_b$ are as in \eqref{eq:Sinfhat.def} and \eqref{eq:Sbhat.def}, respectively.
	
	With a sufficiently large $b_0 > 0$, the $\Lt(\R)$ operators $\widehat \opB_b := ((\widehat \opS_b^* - c) (\widehat \opS_b - c))^{-1}$, $b \in (b_0,\infty]$, are compact, self-adjoint and non-negative. Let $\varsigma_b^2 > 0$ be their spectral radii and $g_b \in \Dom(\widehat \opB_b)$ be corresponding normalised eigenfunctions. Then $g_b \in \Dom(\widehat \opS_b)$ and we have
	\begin{equation*}
		\| (\widehat \opS_b - c) g_b \| = \varsigma_b^{-1} = \| (\widehat \opS_b - c)^{-1} \|^{-1}, \quad b \in (b_0, +\infty].
	\end{equation*}
	Moreover from \eqref{eq:Sbhat.Sinfhat.est} we obtain
	\begin{equation}
		\label{eq:lab.lainf.dist}
		|\varsigma_b - \varsigma_\infty|  = \BigO_K(\varsigma_\infty b^{-1}), \quad b \to +\infty.
	\end{equation}

	Consider $\psi_b \in C_c^{\infty}((-2 \delta_b, 2 \delta_b))$, $0 \le \psi_b \le 1$, $\psi_b = 1$ on $(-\delta_b, \delta_b)$ and such that
	\begin{equation}
		\label{eq:psibp.est}
		\| \psi_b^{(j)} \|_{\infty} \ls (\delta_b)^{-j}, \quad j \in \{1, 2, \dots, N + 1 + l\},
	\end{equation}
	with $N := \max\{\lceil m_a \rceil, \lceil m_q \rceil\} + 1$ and sufficiently large $l \in \N$ (see Remark~\ref{rmk:a.q.symbolclass} and the statement of Lemma~\ref{lem:pdo.comp}, in particular \eqref{eq:R_N.est}). It is clear that $\psi_b \to 1$ pointwise in $\R$ as $b \to +\infty$.
	
	Next we justify that $\psi_b g_b \in \Dom(\widehat{a})$ and therefore $\psi_b g_b \in \Dom(\widehat \opS_b)$ (see \eqref{eq:Sinfhat.def} and \eqref{eq:Sbhat.def}). Letting $u \in \SchwR$, then $\psi_b u \in \Dom(\widehat{a})$ and using the expansion \eqref{eq:compositionformula} we have
	\begin{equation}
		\label{eq:ahat.psib.expansion}
		\widehat{a} \psi_b u = \psi_b \widehat{a} u + [\widehat{a}, \psi_b] u = \psi_b \widehat{a} u + \sum_{j=1}^{N} \frac{i^j}{j!} \psi_b^{(j)} \widehat{a}^{(j)} u + \opR_{N+1} u
	\end{equation}
	and hence, applying Assumption~\ref{asm:a.q.dwe}~\ref{itm:a.symbolclass}, \eqref{eq:psibp.est} and \eqref{eq:R_N.est}, there exists $C > 0$, independent of $b$, such that
	\begin{equation*}
		\begin{aligned}
			\| \widehat{a} \psi_b u \| &\le \| \widehat{a} u \| + \sum_{j=1}^{N} \frac1{j!} \| \psi_b^{(j)} \|_{\infty} \| a^{(j)} \check{u} \| + \| \opR_{N+1} u \|\\
			&\le \| \widehat{a} u \| + C b^{-1} (\| \widehat{a} u \| + \| u \|).
		\end{aligned}
	\end{equation*}
	But $\SchwR$ is a core for $a$, and hence for $\widehat{a}$, and it therefore follows that
	\begin{equation*}
		\| \widehat{a} \psi_b g_b \| \le \| \widehat{a} g_b \| + C b^{-1} (\| \widehat{a} g_b \| + \| g_b \|).
	\end{equation*}
	Since $g_b \in \Dom(\widehat{a})$, this shows that $\psi_b g_b \in \Dom(\widehat{a})$.
	
	Furthermore
	\begin{equation*}
		(\widehat \opS_b - c) \psi_b g_b = (\widehat \opS_b - c) g_b + (\psi_b - 1) (\widehat \opS_b - c) g_b + [\widehat{a} + (2 \la)^{-1} \widehat{q}, \psi_b] g_b.
	\end{equation*}
	Our next goal is to estimate the second and third terms in the above equality. Employing \eqref{eq:lab.lainf.dist}, \eqref{eq:Sbhatinv.graphnorm.est} (and analogously for the adjoint $\widehat S_b^* - c$) and expansions for $[\widehat{a}, \psi_b]$ and $[\widehat{q}, \psi_b]$ such as \eqref{eq:ahat.psib.expansion}, we obtain as $b \to +\infty$
	\begin{equation*}
		\begin{aligned}
			\| (\psi_b - 1) (\widehat \opS_b - c) g_b \| &\ls \| (\psi_b - 1) \xi^{-1} \|_\infty \| \xi (\widehat S_b^* - c)^{-1} \| \| (\widehat S_b^* - c) (\widehat S_b - c) g_b \|\\
			&\ls b^{-1} (1 + \langle c \rangle \| (\widehat \opS_{\infty} - c)^{-1} \|) \varsigma_b^{-2}\\
			&\ls_K b^{-1} \varsigma_{\infty}^{-1},\\
			\| [\widehat{a} + (2 \la)^{-1} \widehat{q}, \psi_b] g_b \| &\ls b^{-1} (\| \widehat{a} g_b \| + \| g_b \|) + b^{-2} (\| \widehat{q} g_b \| + \| g_b \|)\\
			&\ls b^{-1} (\| \widehat{a} g_b \| + \| g_b \|) \ls b^{-1} (\| (\widehat \opS_b - c) g_b \| + \langle c \rangle \| g_b \|)\\
			&\ls b^{-1} (\varsigma_b^{-1} + \langle c \rangle) \ls_K b^{-1} \varsigma_{\infty}^{-1},
		\end{aligned}
	\end{equation*}
	where in the two estimates before the last line we have also used Assumption~\ref{asm:a.q.dwe}~\ref{itm:a.gt.q} and \eqref{eq:Sbhat.graphnorm.est}. Hence $\| (\widehat \opS_b - c) \psi_b g_b \| = \varsigma_b^{-1} + \BigO_K(\varsigma_{\infty}^{-1} b^{-1})$ as $b \to +\infty$. Writing $\psi_b g_b = g_b + (\psi_b - 1) g_b$, we similarly obtain $\| \psi_b g_b \| = 1 + \BigO_K(\varsigma_{\infty}^{-1} b^{-1})$ as $b \to +\infty$. Thus applying \eqref{eq:lab.lainf.dist}, we arrive at
	\begin{equation*}
		\left| \frac{\| (\widehat \opS_b - c) \psi_b g_b \|}{\| \psi_b g_b \|} - \frac1{\varsigma_\infty} \right| = \BigO_K(\varsigma_{\infty}^{-1} b^{-1}), \quad b \to +\infty.
	\end{equation*}

	Recalling that $\widehat \opS_b - c = (2 \la)^{-1} \opU_b \widetilde \opT(\la) \opU_b^{-1}$ and letting $u_b := \opU_b^{-1} \psi_b g_b$, then $u_b \in \Dom(\widehat \opT(\la))$ with $\supp u_b \subset \Omega_{b,+}$. We therefore conclude
	\begin{equation*}
		\left| \frac{(2 |\la|)^{-1} \| \widehat \opT(\la) u_b \|}{\| u_b \|} - \frac1{\varsigma_\infty} \right| =  \BigO_K(\varsigma_{\infty}^{-1} b^{-1}), \quad b \to +\infty
	\end{equation*}
	and the claim follows.
\end{proof}

\subsubsection{Step 4: combining the estimates}
\label{sssec:step.4.R}
With $\Omega_{b,\pm}'$, $\Omega_{b,\pm}$ and $\delta_b$ as defined in \eqref{eq:deltab.def}, \eqref{eq:Omega.def}, let $\phi_{b,\pm} \in C_c^{\infty}(\Omega_{b,\pm})$, $0 \le \phi_{b,\pm} \le 1$, be such that
\begin{equation}
	\label{eq:phib.def}
	\phi_{b,\pm}(\xi) = 1, \; \xi \in \Omega'_{b,\pm}, \quad \|\phi_{b,\pm}^{(j)}\|_{\infty} \ls \delta_b^{-j}, \quad j \in \{1, 2, \dots, N + 1 + l\},
\end{equation}		
with $N := \max\{\lceil m_a \rceil, \lceil m_q \rceil\} + 1$ and sufficiently large $l \in \N$ (see Remark~\ref{rmk:a.q.symbolclass}, the statement of Lemma~\ref{lem:pdo.comp} and, in particular, the upper estimate~\eqref{eq:R_N.est}) and define
\begin{equation}
	\label{eq:phibk.def}
	\begin{aligned}
		\phi_{b,0}(\xi) &:= 1 - (\phi_{b,+}(\xi) + \phi_{b,-}(\xi)), \quad \phi_{b,1}(\xi) := \phi_{b,+}(\xi),\\
		\phi_{b,2}(\xi) &:= \phi_{b,-}(\xi), \quad \xi \in \R.		
	\end{aligned}
\end{equation}
\begin{lemma}
	\label{lem:Tlahat.phibk.commutator}
	Let the assumptions of Theorem~\ref{thm:resolvent.Tla} hold, with $\widehat{a}$ and $\widehat{q}$ as defined in \eqref{eq:Tlahat.def}, and let $\phi_{b,k}$, $k \in \{0, 1, 2\}$, be as defined in \eqref{eq:phibk.def}. Then for all $u \in \SchwR$, all $k \in \{0, 1, 2\}$, we have
	\begin{equation}
		\label{eq:Tlahatphibk.commutator.norm.est}
		\| [\widehat{q} + 2 \la \widehat{a}, \phi_{b,k}] u \| \ls_{\delta} b^{-1} \| \widehat \opT(\la) u \| + \| u \|, \quad b \to +\infty.
	\end{equation}
\end{lemma}
\begin{proof}
	Let $u \in \SchwR$ and $k \in \{0, 1, 2\}$, then by Lemma~\ref{lem:pdo.comp}
	\begin{equation}
		\label{eq:ahat.phib.commutator}
		[\widehat{a}, \phi_{b,k}] u = \sum_{j=1}^{N} \frac{i^j}{j!} \phi_{b,k}^{(j)} \widehat{a}^{(j)} u + \opR_{a,N+1,k} u.
	\end{equation}
	Note that, since $N \ge 2$, we have (see \eqref{eq:R_N.est} and \eqref{eq:phib.def})
	\begin{equation}
		\label{eq:ahat.phib.commutator.est.1}
		\| \opR_{a,N+1,k} u \| \ls b^{-3} \| u \|, \quad b \to +\infty.
	\end{equation}
	Moreover, using Assumption~\ref{asm:a.q.dwe}~\ref{itm:a.symbolclass} with $n = j$, \eqref{eq:phib.def} and \eqref{eq:Tla.graphnorm}, we find for $2 \le j \le N$ and $b \to +\infty$
	\begin{equation}
		\label{eq:ahat.phib.commutator.est.2}
		\begin{aligned}
			\| \phi_{b,k}^{(j)} \widehat{a}^{(j)} u \| &\le \| \phi_{b,k}^{(j)} \|_{\infty} \| \widehat{a}^{(j)} u \| \ls b^{-2} \| (1 + \widehat{a}) u \|\\
			&\ls b^{-2} (b^{-1} \| \widehat \opT(\la) u \| + b \| u \|).
		\end{aligned}
	\end{equation}

	In order to estimate $\| \phi_{b,k}' \widehat{a}^{(1)} u \|$, let us write $\phi_{b,k}' \widehat{a}^{(1)} u = \widehat{a}^{(1)} \phi_{b,k}' u - [\widehat{a}^{(1)}, \phi_{b,k}'] u$. Using Assumption~\ref{asm:a.q.dwe}~\ref{itm:a.symbolclass} with $n = 1$ and \eqref{eq:phib.def} with $j = 1$, we deduce
	\begin{equation*}
		\| \widehat{a}^{(1)} \phi_{b,k}' u \| \ls \| (1 + \widehat{a}) \phi_{b,k}' u \| \ls b^{-1} \| u \| + \| \widehat{a} \phi_{b,k}' u \|.
	\end{equation*}
	Furthermore, noting that $\supp \phi_{b,k}' u \cap (\Omega'_{b,+} \cup \Omega'_{b,-}) = \emptyset$ and applying \eqref{eq:Tlahat.away.est.5}, we obtain as $b \to +\infty$
	\begin{equation*}
		\begin{aligned}
			\| \widehat{a} \phi_{b,k}' u \| &\ls_{\delta} b^{-1} \| \widehat \opT(\la) \phi_{b,k}' u \| \ls_{\delta} b^{-1} (\| \phi_{b,k}' \widehat \opT(\la) u \| + \| [\widehat \opT(\la), \phi_{b,k}'] u \|)\\
			&\ls_{\delta} b^{-1} (b^{-1} \| \widehat \opT(\la) u \| + \| [\widehat{q}, \phi_{b,k}'] u \| + b \| [\widehat{a}, \phi_{b,k}'] u \|).
		\end{aligned}
	\end{equation*}
	Applying Lemma~\ref{lem:pdo.comp}, Assumption~\ref{asm:a.q.dwe}~\ref{itm:a.symbolclass} with $n = j$, \eqref{eq:phib.def} and \eqref{eq:Tla.graphnorm}, we have
	\begin{equation*}
		\begin{aligned}
			\| [\widehat{a}, \phi'_{b,k}] u \| &\le \sum_{j=1}^{N} \frac1{j!} \| \phi_{b,k}^{(j+1)} \widehat{a}^{(j)} u \| + \| \opR'_{a,N+1,k} u \| \ls b^{-2} \| (1 + \widehat{a}) u \| + b^{-4} \| u \|\\
			&\ls b^{-2} (b^{-1} \| \widehat \opT(\la) u \| + b \| u \|), \quad b \to +\infty.
		\end{aligned}
	\end{equation*}
	Moreover, since $1 + q \ls 1 + a$ (by Asumption~\ref{asm:a.q.dwe}~\ref{itm:a.gt.q}), we can similarly derive
	\begin{equation*}
		\| [\widehat{q}, \phi'_{b,k}] u \| \ls b^{-2} (b^{-1} \| \widehat \opT(\la) u \| + b \| u \|), \quad b \to +\infty.
	\end{equation*}
	Therefore
	\begin{equation}
		\label{eq:ahat.phib.commutator.est.3}
		\begin{aligned}
			\| \widehat{a}^{(1)} \phi_{b,k}' u \| &\ls_{\delta} b^{-1} (\| u \| + b^{-1} \| \widehat \opT(\la) u \| + b^{-2} \| \widehat \opT(\la) u \| + \| u \|)\\
			&\ls_{\delta} b^{-1} (b^{-1} \| \widehat \opT(\la) u \| + \| u \|), \quad b \to +\infty.
		\end{aligned}
	\end{equation}
	Expanding the term $[\widehat{a}^{(1)}, \phi_{b,k}'] u$ as before, we have as $b \to +\infty$
	\begin{equation}
		\label{eq:ahat.phib.commutator.est.4}
		\begin{aligned}
			\| [\widehat{a}^{(1)}, \phi'_{b,k}] u \| &\le \sum_{j=1}^{N} \frac1{j!} \| \phi_{b,k}^{(j+1)} \widehat{a}^{(j+1)} u \| + \| \opR''_{a,N+1,k} u \|\\
			&\ls b^{-2} \| (1 + \widehat{a}) u \| + b^{-4} \| u \| \ls b^{-2} (b^{-1} \| \widehat \opT(\la) u \| + b \| u \|).
		\end{aligned}
	\end{equation}
	Hence, combining \eqref{eq:ahat.phib.commutator.est.3} and \eqref{eq:ahat.phib.commutator.est.4}, we obtain
	\begin{equation}
		\label{eq:ahat.phib.commutator.est.5}
		\| \phi_{b,k}' \widehat{a}^{(1)} u \| \ls_{\delta} b^{-1} (b^{-1} \| \widehat \opT(\la) u \| + \| u \|), \quad b \to +\infty.
	\end{equation}

	Substituting estimates \eqref{eq:ahat.phib.commutator.est.5}, \eqref{eq:ahat.phib.commutator.est.2} and \eqref{eq:ahat.phib.commutator.est.1} in \eqref{eq:ahat.phib.commutator}, we conclude
	\begin{equation}
		\label{eq:ahat.phib.commutator.est.6}
		\| [\widehat{a}, \phi_{b,k}] u \| \ls_{\delta} b^{-1} (b^{-1} \| \widehat \opT(\la) u \| + \| u \|), \quad b \to +\infty.
	\end{equation}

	Note that repeating the above process for $[\widehat{q}, \phi_{b,k}] u$, and using $1 + q \ls 1 + a$ from Assumption~\ref{asm:a.q.dwe}~\ref{itm:a.gt.q}, we similarly find
	\begin{equation}
		\label{eq:qhat.phib.commutator.est.1}
		\| [\widehat{q}, \phi_{b,k}] u \| \ls_{\delta} b^{-1} (b^{-1} \| \widehat \opT(\la) u \| + \| u \|), \quad b \to +\infty.
	\end{equation}

	The conclusion \eqref{eq:Tlahatphibk.commutator.norm.est} follows from \eqref{eq:ahat.phib.commutator.est.6} and \eqref{eq:qhat.phib.commutator.est.1}.
\end{proof}
\begin{lemma}
	\label{lem:Tlahat.u1u2.quasiorthogonal}
	Let the assumptions of Theorem~\ref{thm:resolvent.Tla} hold and let $\widehat \opT(\la)$ and  $\phi_{b,k}$, $k \in \{1, 2\}$, be as defined in \eqref{eq:Tlahat.def} and \eqref{eq:phibk.def}, respectively. Then for all $u \in \SchwR$, we have as $b \to +\infty$
	\begin{equation}
		\label{eq:Tlahat.u1u2.quasiorth.norm.est}
		(\| \widehat \opT(\la) \phi_{b,1} u \|^2 + \| \widehat \opT(\la) \phi_{b,2} u \|^2)^{\frac12} = \| \widehat \opT(\la) (\phi_{b,1} + \phi_{b,2}) u \| + \BigO_{\delta}(b^{-1}) (\| \widehat \opT(\la) u \| + \| u \|).
	\end{equation}
\end{lemma}
\begin{proof}
	Let $u \in \SchwR$ and $u_k := \phi_{b,k} u$ with $k \in \{1,2\}$. Applying \eqref{eq:compositionformula} to $[\widehat{q}, \phi_{b,k}]$ and $[\widehat{a}, \phi_{b,k}]$, we have for any $k \in \{1, 2\}$
	\begin{equation*}
		\begin{aligned}
			\widehat \opT(\la) u_k &= \phi_{b,k} \widehat \opT(\la) u + [\widehat{q} + 2 \la \widehat{a}, \phi_{b,k}] u\\
			&= \opB_{N,k}(\la) u + \opR_{N+1,k}(\la) u,
		\end{aligned}
	\end{equation*}
	with
	\begin{align*}
		\opB_{N,k}(\la) u &:= \phi_{b,k} \widehat \opT(\la) u + \sum_{j=1}^{N} \frac{i^j}{j!} \phi_{b,k}^{(j)} (\widehat{q}^{(j)} + 2 \la \widehat{a}^{(j)}) u,\\
		\opR_{N+1,k}(\la) u &:= \opR_{q,N+1,k} u + 2  \la \opR_{a,N+1,k} u.
	\end{align*}
	The remainders $\opR_{q,N+1,k} u$, $\opR_{a,N+1,k} u$ for $[\widehat{q}, \phi_{b,k}]$, $[\widehat{a}, \phi_{b,k}]$, respectively, are defined in \eqref{eq:comp.operator.remainder}. Noting that $\opB_{N,1}(\la) u \subset \Omega_{b,+}$, $\opB_{N,2}(\la) u \subset \Omega_{b,-}$, and consequently $\opB_{N,1}(\la) u \perp \opB_{N,2}(\la) u$ in $L^2$, we deduce
	\begin{equation*}
		\begin{aligned}
			\| \widehat \opT(\la) (u_1 + u_2) \|^2 &= \| \widehat \opT(\la) u_1 \|^2 + \| \widehat \opT(\la) u_2 \|^2 + 2 \Re \langle \opB_{N,1}(\la) u, \opR_{N+1,2}(\la) u \rangle\\
			&\quad + 2 \Re \langle \opR_{N+1,1}(\la) u, \opB_{N,2}(\la) u \rangle\\
			&\quad + 2 \Re \langle \opR_{N+1,1}(\la) u, \opR_{N+1,2}(\la) u \rangle.
		\end{aligned}
	\end{equation*}
	Hence
	\begin{equation}
		\label{eq:Tlahat.u1u2.quasiorth.est.1}
		\begin{aligned}
			&|(\| \widehat \opT(\la) u_1 \|^2 + \| \widehat \opT(\la) u_2 \|^2)^{\frac12} - \| \widehat \opT(\la) (u_1 + u_2) \||\\
			&\qquad\qquad\ls \| \opB_{N,1}(\la) u \|^{\frac12} \| \opR_{N+1,2}(\la) u \|^{\frac12} + \| \opR_{N+1,1}(\la) u \|^{\frac12} \| \opB_{N,2}(\la) u \|^{\frac12}\\
			&\qquad\qquad\quad + \| \opR_{N+1,1}(\la) u \|^{\frac12} \| \opR_{N+1,2}(\la) u \|^{\frac12}.
		\end{aligned}
	\end{equation}

	Since \eqref{eq:ahat.phib.commutator.est.1} holds for $\opR_{q,N+1,k} u$ and $\opR_{a,N+1,k} u$, we find for $k \in \{1, 2\}$
	\begin{equation}
		\label{eq:RNla.norm.est}
		\| \opR_{N+1,k}(\la) u \| \ls b^{-2} \| u \|, \quad b \to +\infty. 
	\end{equation}
	Moreover
	\begin{equation*}
		\| \opB_{N,k}(\la) u - \phi_{b,k} \widehat \opT(\la) u \| \le \| \phi_{b,k}' (\widehat{q}^{(1)} + 2 \la \widehat{a}^{(1)}) u \| + \sum_{j=2}^{N} \frac1{j!} \| \phi_{b,k}^{(j)} (\widehat{q}^{(j)} + 2 \la \widehat{a}^{(j)}) u \|.
	\end{equation*}
	The terms in the right-hand side of the above inequality have already been estimated in Lemma~\ref{lem:Tlahat.phibk.commutator} (see \eqref{eq:ahat.phib.commutator.est.5}, \eqref{eq:ahat.phib.commutator.est.2} and the comments regarding $q$ at the end of the proof). Hence for $k \in \{1, 2\}$
	\begin{equation}
		\label{eq:BNla.norm.est}
		|\| \opB_{N,k}(\la) u \| - \| \widehat \opT(\la) u \|| \ls_{\delta} b^{-1} \| \widehat \opT(\la) u \| + \| u \|, \quad b \to +\infty.
	\end{equation}

	Applying \eqref{eq:RNla.norm.est} and \eqref{eq:BNla.norm.est}, we can estimate the first term in the right-hand side of \eqref{eq:Tlahat.u1u2.quasiorth.est.1} as $b \to +\infty$
	\begin{equation*}
		\begin{aligned}
			\| \opB_{N,1}(\la) u \|^{\frac12} \| \opR_{N+1,2}(\la) u \|^{\frac12} &\ls b^{-1} \| \opB_{N,1}(\la) u \| + b \| \opR_{N+1,2}(\la) u \|\\
			&\ls_{\delta} b^{-1} (\| \widehat \opT(\la) u \| + \| u \|).
		\end{aligned}
	\end{equation*}
	A similar estimate can be derived for $\| \opB_{N,2}(\la) u \|^{\frac12} \| \opR_{N+1,1}(\la) u \|^{\frac12}$ which, combined with \eqref{eq:RNla.norm.est}, yields the desired result.
\end{proof}
\begin{proof}[Proof of Theorem~\ref{thm:resolvent.Tla}]
	Let $0 \ne u \in \SchwR \subset \Dom(\widehat \opT(\la))$ and let us write $u = u_0 + u_1 + u_2$, where $u_k := \phi_{b,k} u$ with $k \in \{0,1,2\}$ and $\phi_{b,k}$ as defined in \eqref{eq:phibk.def}. Then
	\begin{equation*}
		\widehat \opT(\la) u_k = \phi_{b,k} \widehat \opT(\la) u + [\widehat{q} + 2 \la \widehat{a}, \phi_{b,k}] u, \quad k \in \{0,1,2\},
	\end{equation*}
	and therefore, noting that $\supp \phi_{b,1} \cap \supp \phi_{b,2} = \emptyset$ and applying Lemma~\ref{lem:Tlahat.phibk.commutator}, we obtain as $b \to +\infty$
	\begin{equation}
		\label{eq:Tlahat.uk.norm.est}
		\begin{aligned}
			\| \widehat \opT(\la) u_0 \| &\le (1 + \BigO_{\delta}(b^{-1})) \| \widehat \opT(\la) u \| + \BigO_{\delta}(1) \| u \|,\\
			\| \widehat \opT(\la) (u_1 + u_2) \| &\le (1 + \BigO_{\delta}(b^{-1})) \| \widehat \opT(\la) u \| + \BigO_{\delta}(1) \| u \|.
		\end{aligned}
	\end{equation}

	Firstly, using the fact that $u_1 \perp u_2$ in combination with Proposition~\ref{prop:local.Tla} and Lemma~\ref{lem:Tlahat.u1u2.quasiorthogonal}, we find as $b \to +\infty$
	\begin{equation}
		\label{eq:uk.local.est1}
		\begin{aligned}
			\| u_1 + u_2 \| &\le \| (\opA - c)^{-1} \| (2 b)^{-1} (1 + \BigO_K(b^{-1})) (\| \widehat \opT(\la) u_1 \|^2 + \| \widehat \opT(\la) u_2 \|^2)^{\frac12}\\
			&\le \| (\opA - c)^{-1} \| (2 b)^{-1} (1 + \BigO_K(b^{-1})) (\| \widehat \opT(\la) (u_1 + u_2) \|\\
			&\quad + \BigO_{\delta}(b^{-1}) (\| \widehat \opT(\la) u \| + \| u \|)).
		\end{aligned}
	\end{equation}
	Thus by \eqref{eq:Tlahat.uk.norm.est} we have as $b \to +\infty$
	\begin{equation}
		\label{eq:uk.local.est}
		\| u_1 + u_2 \| \le \| (\opA - c)^{-1} \| (2 b)^{-1} (1 + \BigO_K(b^{-1})) \| \widehat \opT(\la) u \| + \BigO_K(b^{-1}) \| u \|.
	\end{equation}

	Secondly, since $\supp u_0 \cap (\Omega'_{b,+} \cup \Omega'_{b,-}) = \emptyset$, then by Proposition~\ref{prop:away.Tla}
	\begin{equation*}
		b^2 \| u_0 \| \ls_{\delta} \| \widehat \opT(\la) u_0 \|, \quad b \to +\infty,
	\end{equation*}
	and applying \eqref{eq:Tlahat.uk.norm.est} we have
	\begin{equation}
		\label{eq:uk.away.est}
		\| u_0 \| \ls_{\delta} b^{-2} (\| \widehat \opT(\la) u \| + \| u \|), \quad b \to +\infty.
	\end{equation}

	Combining \eqref{eq:uk.local.est} and \eqref{eq:uk.away.est}, we find that for $b \to +\infty$
	\begin{equation*}
		\| u \| \le \| (\opA - c)^{-1} \| (2 b)^{-1} (1 + \BigO_K(b^{-1})) \| \widehat \opT(\la) u \| + \BigO_K(b^{-1}) \| u \|
	\end{equation*}
	and therefore
	\begin{equation}
		\label{eq:Hb.est.R}
		\| u \| \le \| (\opA - c)^{-1} \| (2 b)^{-1} (1 + \BigO_K(b^{-1})) \| \widehat \opT(\la) u \|.
	\end{equation}

	Since $\SchwR$ is a core for $\opT(\la)$, and equivalently for $\widehat \opT(\la)$, we can extend the above estimate to any $u \in \Dom(\widehat \opT(\la))$ relying on standard approximation arguments. The proof of the theorem follows by an appeal to Proposition~\ref{prop:lbound.Tla} and the use of the inverse Fourier transform to take the result back to $x$-space.
\end{proof}

\subsection{The norm of the resolvent along curves adjacent to the imaginary axis}
\label{ssec:resnorm.adj.dwe}
As in the analysis for Schr\"odinger operators with complex potential carried out in \cite[Sub-section~5.1]{ArSi-resolvent-2022}, it is possible to extend the proof of Theorem~\ref{thm:resolvent.Tla} to more general curves inside the left-hand side semi-plane $\overline\C_{-}$
\begin{equation}
	\label{eq:lab.Tla.def}
	\la_b := -c(b) + i b,
\end{equation}
where $b \in \R \setminus \{0\}$ and $c : \R \setminus \{0\} \to \overline\Rplus$ satisfies
\begin{gather}
	\label{eq:gral.curves.Tla.adj.asm} c_b |b|^{-1} = o(1), \quad |b| \to +\infty,\\
	\label{eq:gral.curves.Tla.Phib.asm} \Phi_b := \langle c_b \rangle^2 \| (\opA - c_b)^{-1} \| |b|^{-1} = o(1), \quad |b| \to +\infty,
\end{gather}
with $\opA$ as defined in \eqref{eq:Airy.def} and $c_b \equiv c(b)$. We are interested in two types of curves:
\begin{enumerate}[\upshape (1)]
	\item $\la_b$ with $c_b$ satisfying
	\begin{equation}
		\label{eq:gral.curves.dwe.straight.lines}
		c_b \ls 1, \quad b \to +\infty;
	\end{equation}
	\item $\la_b$ with $c_b$ satisfying
	\begin{equation}
		\label{eq:gral.curves.dwe.unbd.curves}
		\langle c_b \rangle \| (\opA - c_b)^{-1} \| \to +\infty, \quad b \to +\infty.
	\end{equation}
\end{enumerate}	
Note that, when \eqref{eq:gral.curves.dwe.straight.lines} holds (\eg~in the statement of Theorem~\ref{thm:resolvent.Tla}), we have $\langle c_b \rangle^2 \| (\opA - c_b)^{-1} \| \ls 1$ and therefore conditions \eqref{eq:gral.curves.Tla.adj.asm}-\eqref{eq:gral.curves.Tla.Phib.asm} are both automatically satisfied.

We also observe that, because of Assumption~\eqref{eq:gral.curves.Tla.adj.asm}, we have $\xi^2 + \la^2 = \xi^2 - b^2 (1 + o(1))$ when $|b| \to +\infty$, as in the proof of Theorem~\ref{thm:resolvent.Tla}.
\begin{proposition}
	\label{prop:gral.curves.Tla}
	Let $a$ and $q$ satisfy Assumption~\ref{asm:a.q.dwe} and let $\opT(\la_b)$ be the family of operators \eqref{eq:Tla.def}-\eqref{eq:Tla.domain} for $\la_b$ defined by \eqref{eq:lab.Tla.def}. Assume furthermore that \eqref{eq:gral.curves.Tla.adj.asm}-\eqref{eq:gral.curves.Tla.Phib.asm} hold with $c_b$ satisfying either \eqref{eq:gral.curves.dwe.straight.lines} or \eqref{eq:gral.curves.dwe.unbd.curves}. Then
	\begin{equation}
		\label{eq:gral.curves.Tla}
		\| \opT(\la_b)^{-1} \| = \| (\opA - c_b)^{-1} \| (2 |b|)^{-1} (1 + \BigO(\Phi_b)), \quad |b| \to +\infty,
	\end{equation}
	with $\opA$ as defined in \eqref{eq:Airy.def}.
\end{proposition}
\begin{proof}[Sketch of proof]
	We shall closely follow the steps in Sub-section~\ref{ssec:proof.resolvent.Tla}, keeping the notation introduced there but omitting details whenever the arguments used earlier remain valid. As before, it is enough for us to consider the case $b > 0$, $b \to +\infty$. We note that, for families $\la_b$ satisfying \eqref{eq:gral.curves.dwe.straight.lines} or \eqref{eq:gral.curves.dwe.unbd.curves}, the choice of parameter $\delta$ in \eqref{eq:deltab.def} is independent of $b$ (see Step 2 below).
	
	\underline{Step 1}
	
	Assumption \eqref{eq:gral.curves.Tla.adj.asm} is enough to ensure that \eqref{eq:Tlahat.away.est.5} continues to hold for $\la_b$ and hence we have as $b \to +\infty$
	\begin{equation*}
		b^2 \ls \inf \left\{ \frac{\| \widehat\opT(\la_b)  u \|}{\|u\|} : \; 0 \neq u \in \Dom(\widehat\opT(\la_b)), \; \supp u \cap (\Omega'_{b,+} \cup \Omega'_{b,-}) = \emptyset \right\}.
	\end{equation*}

	\underline{Step 2}
	
	We use the notation in Proposition~\ref{prop:local.Tla}, replacing $\la$ with $\la_b$ and $c$ with $c_b$ where necessary. From \eqref{eq:U.That.Uminus1}, \eqref{eq:Sinfhat.def} and \eqref{eq:Sbhat.def}, we have
	\begin{equation}
		\label{eq:U.That.Uminus1.gralcurves}
		\widehat\opS_b - c_b = \frac1{2 \la_b}\opU_b \widetilde\opT(\la_b) \opU_b^{-1} = \widehat\opS_{\infty} - c_b + \frac1{2 \la_b} \widehat{q} + \widehat\opR_b,
	\end{equation}
	with $\widehat\opR_b(\xi)$ as defined in \eqref{eq:Rbhat.def}. Our next aim is to prove that $c_b \in \rho(\widehat\opS_b)$ as $b \to +\infty$. To do this, we argue as in Step 2 of \cite[Prop.~5.1]{ArSi-resolvent-2022}. For any $c_b > 0$, the operator $\widehat\opK_{b,\infty} := \opI - c_b \widehat\opS_{\infty}^{-1} = \widehat\opS_{\infty}^{-1} (\widehat\opS_{\infty} - c_b) = (\widehat\opS_{\infty} - c_b) \widehat\opS_{\infty}^{-1}$ is bounded and invertible and moreover by \eqref{eq:Sinfhatinv.graphnorm.est.1} (note also that either \eqref{eq:gral.curves.dwe.straight.lines} or \eqref{eq:gral.curves.dwe.unbd.curves} holds by assumption) we have for $b \to +\infty$
	\begin{equation}
		\label{eq:Kbhatiminus1.dwe.est}
		\| \widehat\opK_{b,\infty}^{-1} \| \ls \langle c_b \rangle \| (\opA - c_b)^{-1} \|.
	\end{equation}
	Recalling from Proposition~\ref{prop:local.Tla} that $0 \in \rho(\widehat\opS_b)$ for large enough $b$ and defining $\widehat\opK_b := \opI - c_b \widehat\opS_b^{-1} = \widehat\opS_b^{-1} (\widehat\opS_b - c_b) = (\widehat\opS_b - c_b) \widehat\opS_b^{-1}$, we find
	\begin{equation*}
		\widehat\opK_b = \widehat\opK_{b,\infty} (\opI - c_b \widehat\opK_{b,\infty}^{-1} (\widehat\opS_b^{-1} - \widehat\opS_\infty^{-1})).
	\end{equation*}
	Moreover, by \eqref{eq:Sbhat.Sinfhat.est} with $c = 0$, \eqref{eq:Kbhatiminus1.dwe.est} and \eqref{eq:gral.curves.Tla.Phib.asm}, we have
	\begin{equation*}
		\| c_b \widehat\opK_{b,\infty}^{-1} (\widehat\opS_b^{-1} - \widehat\opS_\infty^{-1}) \| \ls \Phi_b = o(1), \quad b \to +\infty.
	\end{equation*}
	It follows that $\widehat\opK_b$ is invertible and $\| \widehat\opK_b^{-1} \| \approx \| \widehat\opK_{b,\infty}^{-1} \|$ as $b \to +\infty$. Since $\widehat\opS_b - c_b = \widehat\opK_b \widehat\opS_b = \widehat\opS_b \widehat\opK_b$, we conclude that $c_b \in \rho(\widehat\opS_b)$ for $b \to +\infty$, as claimed. Moreover, $(\widehat\opS_b - c_b)^{-1} = \widehat\opS_b^{-1} \widehat\opK_b^{-1} = \widehat\opK_b^{-1} \widehat\opS_b^{-1}$ and, applying \eqref{eq:Sbhatinv.graphnorm.est} with $c = 0$ and \eqref{eq:Kbhatiminus1.dwe.est}, we deduce as $b \to +\infty$
	\begin{equation}
		\label{eq:xiSbhat.normres.dwe.gralcurves}
		\| \xi (\widehat\opS_b - c_b)^{-1} \| + \| \xi (\widehat\opS_b^* - c_b)^{-1} \| \ls \langle c_b \rangle \| (\opA - c_b)^{-1} \|.
	\end{equation}

	Furthermore, we have (see the argument in \cite[Eq.~(5.15)]{ArSi-resolvent-2022})
	\begin{equation*}
		((\widehat\opS_b - c_b)^{-1} - (\widehat\opS_\infty - c_b)^{-1}) \widehat\opK_b = \widehat\opK_{b,\infty}^{-1} (\widehat\opS_b^{-1} - \widehat\opS_{\infty}^{-1}).
	\end{equation*}
	Hence
	\begin{equation*}
		(\widehat\opS_b - c_b)^{-1} - (\widehat\opS_\infty - c_b)^{-1} = \widehat\opK_{b,\infty}^{-1} (\widehat\opS_b^{-1} - \widehat\opS_{\infty}^{-1}) \widehat\opK_b^{-1}, \quad b \to +\infty,
	\end{equation*}
	and therefore by \eqref{eq:Sbhat.Sinfhat.est} with $c = 0$ and \eqref{eq:Kbhatiminus1.dwe.est}, we have
	\begin{equation*}
		\| (\widehat\opS_b - c_b)^{-1} - (\widehat\opS_\infty - c_b)^{-1} \| \ls \| (\opA - c_b)^{-1} \| \Phi_b, \quad b \to +\infty.
	\end{equation*}
	It follows that
	\begin{equation}
		\label{eq:Sbhat.Sinfhat.resnorms.dwe.gralcurves}
		\| (\widehat\opS_b - c_b)^{-1} \| = \| (\opA - c_b)^{-1} \| (1 + \BigO(\Phi_b)), \quad b \to +\infty,
	\end{equation}
	and hence from \eqref{eq:U.That.Uminus1.gralcurves} and \eqref{eq:Sbhat.Sinfhat.resnorms.dwe.gralcurves} as $b \to +\infty$
	\begin{equation*}
		2 |\la_b| \| \widetilde\opT(\la_b)^{-1} \| = \| (\widehat\opS_b - c_b)^{-1} \| = \| (\opA - c_b)^{-1} \| (1 + \BigO(\Phi_b)).
	\end{equation*}
	Arguing as in the last stage of Proposition~\ref{prop:local.Tla} and noting that $|\la_b| = b \sqrt{1 + c_b^2 b^{-2}} = b (1 + \BigO(c_b^2 b^{-2}))$ and furthermore, by \eqref{eq:gral.curves.Tla.adj.asm} and \eqref{eq:gral.curves.dwe.straight.lines}-\eqref{eq:gral.curves.dwe.unbd.curves}, we have $c_b^2 b^{-2} \ls \Phi_b$, we deduce as $b \to +\infty$
	\begin{equation}
		\label{eq:resnorm.local.dwe.gralcurve}
		\begin{aligned}
			&\| (\opA - c_b)^{-1} \|^{-1} 2 b (1 - \BigO(\Phi_b))\\
			& \qquad  \qquad \le \inf \left\{ \frac{\left\| \widehat\opT(\la_b) u \right\|}{\|u\|}: \; 0 \neq u \in \Dom(\widehat\opT(\la_b)), \; \supp u \subset \Omega_{b,\pm} \right\}.
		\end{aligned}	
	\end{equation}

	\underline{Step 3}
	
	We follow the proof of Proposition~\ref{prop:lbound.Tla}, replacing $\widehat\opS_b - c$ with $\widehat\opS_b - c_b$, to find $g_b\in \Dom((\widehat\opS_b^* - c_b) (\widehat\opS_b - c_b))$ such that
	\begin{equation}
		\label{eq:varsigb.dwe.def}
		\| (\widehat\opS_b - c_b) g_b \| = \varsigma_b^{-1} = \| (\widehat\opS_b - c_b)^{-1} \|^{-1}, \quad b \to +\infty.
	\end{equation}
	Moreover, with $\varsigma_{b,\infty} := \| (\opA - c_b)^{-1} \|$, we have (see \eqref{eq:Sbhat.Sinfhat.resnorms.dwe.gralcurves})
	\begin{equation}
		\label{eq:lab.conv.dwe.gralcurve}
		\varsigma_b = \varsigma_{b,\infty} (1 + \BigO(\Phi_b)), \quad b \to +\infty.
	\end{equation}

	Recalling the cut-off functions $\psi_b$, we write
	\begin{equation}
		\label{eq:Sbhat.cb.psib.gb}
		(\widehat \opS_b - c_b) \psi_b g_b = (\widehat \opS_b - c_b) g_b + (\psi_b - 1) (\widehat \opS_b - c_b) g_b + [\widehat{a} + (2 \la)^{-1} \widehat{q}, \psi_b] g_b
	\end{equation}
	and we proceed to estimate as before the second and third terms in the right-hand side of the above equality, using also \eqref{eq:xiSbhat.normres.dwe.gralcurves}, \eqref{eq:gral.curves.dwe.straight.lines}-\eqref{eq:gral.curves.dwe.unbd.curves} and \eqref{eq:lab.conv.dwe.gralcurve}, for $b \to +\infty$
	\begin{equation*}
		\begin{aligned}
			\| (\psi_b - 1) (\widehat\opS_b - c_b) g_b \| &\ls b^{-1} \langle c_b \rangle \| (\opA - c_b)^{-1} \| \varsigma_{b}^{-2} \ls b^{-1} \langle c_b \rangle \varsigma_{b,\infty}^{-1} \ls \Phi_b \langle c_b \rangle^{-1} \varsigma_{b,\infty}^{-2}\\
			\| [\widehat{a} + (2 \la)^{-1} \widehat{q}, \psi_b] g_b \| &\ls b^{-1} (\| (\widehat\opS_b - c_b) g_b \| + \langle c_b \rangle \| g_b \|) \ls b^{-1} \langle c_b \rangle \ls \Phi_b \langle c_b \rangle^{-1} \varsigma_{b,\infty}^{-1}.
		\end{aligned}
	\end{equation*}
	Hence $\| (\widehat\opS_b - c_b) \psi_b g_b \| = \varsigma_b^{-1} + \BigO(\Phi_b \langle c_b \rangle^{-1} \varsigma_{b,\infty}^{-1}(1 + \varsigma_{b,\infty}^{-1}))$ as $b \to +\infty$. Writing $\psi_b g_b = g_b + (\psi_b - 1) g_b$, we similarly obtain $\| \psi_b g_b \| = 1 + \BigO(\Phi_b \langle c_b \rangle^{-1} \varsigma_{b,\infty}^{-1})$ as $b \to +\infty$. Thus applying \eqref{eq:lab.conv.dwe.gralcurve}, we arrive at
	\begin{equation*}
		\left| \frac{\| (\widehat \opS_b - c_b) \psi_b g_b \|}{\| \psi_b g_b \|} - \frac1{\varsigma_{b,\infty}} \right| = \BigO(\varsigma_{b,\infty}^{-1} \Phi_b), \quad b \to +\infty.
	\end{equation*}
	Recalling that $\widehat\opS_b - c_b = (2 \la_b)^{-1} \opU_b \widetilde\opT(\la_b) \opU_b^{-1}$ and letting $u_b := \opU_b^{-1} \psi_b g_b$, then $u_b \in \Dom(\widehat\opT(\la_b))$ with $\supp u_b \subset \Omega_{b,+}$ and we have
	\begin{equation*}
		\left| \frac{(2 |\la_b|)^{-1} \| \widehat\opT(\la_b) u_b \|}{\| u_b \|} - \frac1{\varsigma_{b,\infty}} \right| =  \BigO(\varsigma_{b,\infty}^{-1}\Phi_b), \quad b \to +\infty.
	\end{equation*}
	Hence
	\begin{equation}
		\label{eq:Tlabhat.low.bound.gral.curves}
		\frac{\| \widehat\opT(\la_b) u_b \|}{\| u_b \|} = \| (\opA - c_b)^{-1} \|^{-1} 2 b (1 + \BigO(\Phi_b)), \quad b \to + \infty.
	\end{equation}

	\underline{Step 4}
	
	It is straightforward to verify that estimates \eqref{eq:Tlahatphibk.commutator.norm.est} in Lemma~\ref{lem:Tlahat.phibk.commutator} and \eqref{eq:Tlahat.u1u2.quasiorth.norm.est} in Lemma~\ref{lem:Tlahat.u1u2.quasiorthogonal} continue to hold when we replace $\la$ with $\la_b$ (with $\delta$ independent of $b$). As in the proof of Theorem~\ref{thm:resolvent.Tla}, we have as $b \to +\infty$
	\begin{equation}
		\label{eq:Tlabhat.uk.norm.est}
		\begin{aligned}
			\| \widehat\opT(\la_b) u_0 \| &\le (1 + \BigO(b^{-1})) \| \widehat\opT(\la_b) u \| + \BigO(1) \| u \|,\\
			\| \widehat\opT(\la_b) (u_1 + u_2) \| &\le (1 + \BigO(b^{-1})) \| \widehat\opT(\la_b) u \| + \BigO(1) \| u \|.
		\end{aligned}
	\end{equation}
	By \eqref{eq:resnorm.local.dwe.gralcurve}, \eqref{eq:Tlahat.u1u2.quasiorth.norm.est} and \eqref{eq:Tlabhat.uk.norm.est}, we obtain for $b \to +\infty$
	\begin{equation}
		\label{eq:uk.gral.curves.dwe.local.est}
		\begin{aligned}
			2 b \| u_1 + u_2 \| &\le \| (\opA - c_b)^{-1} \| (1 + \BigO(\Phi_b)) \| \widehat\opT(\la_b) (u_1 + u_2) \|\\
			&\quad + \BigO(\| (\opA - c_b)^{-1} \| b^{-1}) (\| \widehat\opT(\la_b) u \| + \| u \|)\\
			&\le \| (\opA - c_b)^{-1} \| (1 + \BigO(\Phi_b)) \| \widehat\opT(\la_b) u \| + \BigO(\| (\opA - c_b)^{-1} \|) \| u \|.
		\end{aligned}
	\end{equation}
	By \eqref{eq:Tlahat.away.est.5} and \eqref{eq:Tlabhat.uk.norm.est}, we have as $b \to +\infty$
	\begin{equation}
		\label{eq:uk.gral.curves.dwe.away.est}
		2 b \| u_0 \| \ls b^{-1} (\| \widehat\opT(\la_b) u \| + \| u \|).
	\end{equation}
	Combining \eqref{eq:uk.gral.curves.dwe.local.est} and \eqref{eq:uk.gral.curves.dwe.away.est}, we find that as $b \to +\infty$
	\begin{equation*}
		\begin{aligned}
			2 b \| u \| &\le 2 b \left(\| u_0 \| + \| u_1 + u_2 \| \right)\\
			&\le \| (\opA - c_b)^{-1} \| (1 + \BigO(\Phi_b)) \| \widehat\opT(\la_b) u \| + \BigO(\| (\opA - c_b)^{-1} \|) \| u \|
		\end{aligned}
	\end{equation*}
	and hence
	\begin{equation*}
		2 b (1 - \BigO(\| (\opA - c_b)^{-1} \| b^{-1})) \| u \| \le \| (\opA - c_b)^{-1} \| (1 + \BigO(\Phi_b)) \| \widehat\opT(\la_b) u \|.
	\end{equation*}
	It follows
	\begin{equation*}
		\| u \| \le \| (\opA - c_b)^{-1} \| (2 b)^{-1} (1 + \BigO(\Phi_b)) \| \widehat\opT(\la_b) u \|, \quad b \to +\infty.
	\end{equation*}
	This result combined with the lower bound \eqref{eq:Tlabhat.low.bound.gral.curves} yields \eqref{eq:gral.curves.Tla}.		
\end{proof}

As an application of Proposition~\ref{prop:gral.curves.Tla} and the resolvent norm estimate for generalised Airy operators found in \cite[Thm.~4.2]{ArSi-generalised-2022}, we shall consider an example of damping function that satisfies Assumption~\ref{asm:a.q.dwe} and \cite[Asm.~3.1]{ArSi-generalised-2022} (note that the choice of $q$ plays no role in the calculations provided that Assumptions~\ref{asm:a.q.dwe}~\ref{itm:q.symbolclass}-\ref{itm:a.gt.q} are satisfied).

Let $a(x) = \log \langle x \rangle^p$, $p > 0$. We have (see \cite[Ex.~4.3(i)]{ArSi-generalised-2022})
\begin{gather*}
	\| (\opA - c_b)^{-1} \| = \sqrt{\frac{\pi}{p}} \exp\left(2 p \sqrt{\exp \left(\frac{2c_b}{p} \right) - 1} + \frac{c_b}{2 p} - p \pi\right) (1 + o(1)), \quad b \to +\infty,\\
	\implies \log\log(\| (\opA - c_b)^{-1} \|) = \frac{c_b}{p} (1 + o(1)), \quad b \to +\infty.
\end{gather*}
Using \eqref{eq:gral.curves.Tla} and substituting $\| (\opT(\la_b))^{-1} \| = \eps^{-1}$, with $\eps > 0$, we obtain the level curves
\begin{equation}
	\label{eq:a.log.lev.curves.dwe}
	c_b = p \log\log(2 b \eps^{-1}) (1 + o(1)), \quad b \to +\infty.
\end{equation}
Note that, in terms of assumptions \eqref{eq:gral.curves.Tla.adj.asm}-\eqref{eq:gral.curves.Tla.Phib.asm}, we have $c_b b^{-1} = o(1)$. On the other hand, \eqref{eq:gral.curves.Tla.Phib.asm} does not hold: if it did, by \eqref{eq:gral.curves.Tla}, we would have $\| (\opA - c_b)^{-1} \| b^{-1} \approx 1$ along the level curves and hence $\Phi_b \to +\infty$ as $b \to +\infty$. We shall therefore just put forward as a \textit{conjecture} that \eqref{eq:a.log.lev.curves.dwe} asymptotically describes the level curves for $a(x) = \log \langle x \rangle^p$.

\section{The operator $\opG$}
\label{sec:mainthmG}
\subsection{Proof of Theorem~\ref{thm:resolvent.G}}
\label{ssec:proof.resolvent.G}
\begin{proof}[Proof of Theorem~\ref{thm:resolvent.G}]
	Our first goal is to find a lower bound for $\| (\opG - \la)^{-1} \|$, with $\la := -c + i b$ defined in the statement of the theorem, as $|b| \to +\infty$.
	
	Let us take arbitrary $u_1, u_2 \in \CcR$ and let $u := (u_1, u_2)^t$. Then $u \in \Dom(\opG) \subset \cH$ and
	\begin{equation*}
		\begin{aligned}
			\| u \|_{\cH} &= (\| \Ntp u_1 \|^2 + \| q^{\frac12} u_1 \|^2 + \| u_2 \|^2)^{\frac12},\\
			\| (\opG - \la) u \|_{\cH} &= (\| \Ntp (\la u_1 - u_2) \|^2 + \| q^{\frac12} (\la u_1 - u_2) \|^2 + \| \opH_q u_1 + (2 a + \la) u_2 \|^2)^{\frac12}.
		\end{aligned}
	\end{equation*}
	Choosing $u_2 := \la u_1$ with $u_1 \ne 0$, we have $\| u \|_{\cH} \ge |b| \| u_1 \|$ and $\| (\opG - \la) u \|_{\cH} = \| \opT(\la) u_1 \|$. Noticing that, by the spectral equivalence \eqref{eq:G.Tla.spectral.equivalence}, $(\opG - \la)^{-1}$ exists if and only if $\opT(\la)^{-1}$ exists and that the existence of $\opT(\la)^{-1}$ is guaranteed by Theorem~\ref{thm:resolvent.Tla} for sufficiently large $|b|$, it follows that
	\begin{equation*}
		\frac{\| u \|_{\cH}}{\| (\opG - \la) u \|_{\cH}} \ge |b| \frac{\| u_1 \|}{\| \opT(\la) u_1 \|} \implies \| (\opG - \la)^{-1} \| \ge |b| \| \opT(\la)^{-1} \|,
	\end{equation*}
	where, for the last implication, we have used the fact that $\CcR$ is densely contained in $\Lt(\R)$. An application of Theorem~\ref{thm:resolvent.Tla} shows that $\| (\opG - \la)^{-1} \| \gs_K 1$ as $|b| \to +\infty$.
	
	In order to find an upper bound for $\| (\opG - \la)^{-1} \|$, let $0 \ne v \in \CcR \times \CcR \subset \cH$ and set $u := (\opG - \la)^{-1} v \in \Dom(\opG)$. Using \eqref{eq:Rla.def} and \eqref{eq:Tla.def}, we have
	\begin{equation*}
			u = \begin{pmatrix}
			-\la^{-1} (\opI - \opT(\la)^{-1} \opH_q) & -\opT(\la)^{-1}\\
			\opT(\la)^{-1} \opH_q & -\la \opT(\la)^{-1}
		\end{pmatrix} \begin{pmatrix}
		v_1\\
		v_2
	\end{pmatrix}
	\end{equation*}
	and therefore
	\begin{equation}
		\label{eq:G.resnorm.u1.u2.def}
		\begin{aligned}
			u_1 &= -\la^{-1} (\opI - \opT(\la)^{-1} \opH_q) v_1 - \opT(\la)^{-1} v_2\\
			u_2 &= \opT(\la)^{-1} \opH_q v_1 - \la \opT(\la)^{-1} v_2.
		\end{aligned}
	\end{equation}
	Our next task is to estimate $\| u \|_{\cH}$ and to this end we shall find upper bounds for $\| \opH_q^{\frac12} u_1 \|$ and $\| u_2 \|$, with $u_1$ and $u_2$ as in \eqref{eq:G.resnorm.u1.u2.def}.
	
	Considering firstly $u_2$ and applying \eqref{eq:Hq12.Tlaminus1.norm} and \eqref{eq:resnorm.Tla}, we obtain as $|b| \to +\infty$
	\begin{equation}
		\label{eq:G.resnorm.u2.est}
		\| u_2 \| \ls_K \| \opH_q^{\frac12} v_1 \| + \| v_2 \|.
	\end{equation}

	Turning to $\opH_q^{\frac12} u_1$ we find
	\begin{equation*}
		\opH_q^{\frac12} u_1 = -\la^{-1} (\opH_q^{\frac12} v_1 - \opH_q^{\frac12} \opT(\la)^{-1} \opH_q v_1) - \opH_q^{\frac12} \opT(\la)^{-1} v_2
	\end{equation*}
	(note that $\opH_q$ is a positive self-adjoint operator and $\Dom(\opT(\la)) \subset \Dom(\opH_q)$, hence the above operations make sense). Applying \eqref{eq:Hq12.Tlaminus1.norm}, it follows as $|b| \to +\infty$
	\begin{equation}
		\label{eq:G.resnorm.Hq12u1.est1}
		\| \opH_q^{\frac12} u_1 \| \ls_K |b|^{-1} (\| \opH_q^{\frac12} v_1 \| + \| \opH_q^{\frac12} \opT(\la)^{-1} \opH_q v_1 \|) + \| v_2 \|.
	\end{equation}
	To estimate $\| \opH_q^{\frac12} \opT(\la)^{-1} \opH_q^{\frac12} \|$, we use the second resolvent identity with $\mu > 0$
	\begin{equation*}
		\begin{aligned}
			\opT(\la)^{-1} &= \opT(\mu)^{-1} + (\mu - \la) \opT(\mu)^{-1} (2 a + \mu + \la) \opT(\la)^{-1}\\
			&= \frac{\mu}{\la} \opT(\mu)^{-1} + \mu (\mu - \la) \opT(\mu)^{-1} \opT(\la)^{-1} - \frac{\mu - \la}{\la} \opT(\mu)^{-1} \opH_q \opT(\la)^{-1}.
		\end{aligned}
	\end{equation*}
	Hence
	\begin{equation*}
		\begin{aligned}
			\opH_q^{\frac12} \opT(\la)^{-1} \opH_q^{\frac12} &= \frac{\mu}{\la} \opH_q^{\frac12} \opT(\mu)^{-1} \opH_q^{\frac12} + \mu (\mu - \la) \opH_q^{\frac12} \opT(\mu)^{-1} \opT(\la)^{-1} \opH_q^{\frac12}\\
			&\quad - \frac{\mu - \la}{\la} \opH_q^{\frac12} \opT(\mu)^{-1} \opH_q \opT(\la)^{-1} \opH_q^{\frac12}.
		\end{aligned}
	\end{equation*}
	Letting $z_{\la,\mu} := -\la / (\mu - \la) = (c^2 + c \mu + b^2 - i \mu b) / ((c + \mu)^2 + b^2)$, we deduce
	\begin{equation*}
		\begin{aligned}
			(\opH_q^{\frac12} \opT(\mu)^{-1} \opH_q^{\frac12} - z_{\la,\mu}) \opH_q^{\frac12} \opT(\la)^{-1} \opH_q^{\frac12} &= \frac{\mu}{\mu - \la} \opH_q^{\frac12} \opT(\mu)^{-1} \opH_q^{\frac12}\\
			&\quad + \mu \la \opH_q^{\frac12} \opT(\mu)^{-1} \opT(\la)^{-1} \opH_q^{\frac12}.
		\end{aligned}
	\end{equation*}
	We observe that $\opH_q^{\frac12} \opT(\mu)^{-1} \opH_q^{\frac12}$ is self-adjoint, positive and it can be boundedly extended to $\Lt(\R)$ by \eqref{eq:Hq12.Tmuminus12.norm}. Furthermore, for $b \ne 0$ we have $z_{\la,\mu} \notin \R$ and therefore the operator $\opH_q^{\frac12} \opT(\mu)^{-1} \opH_q^{\frac12} - z_{\la,\mu}$ is invertible and
	\begin{equation*}
		\| (\opH_q^{\frac12} \opT(\mu)^{-1} \opH_q^{\frac12} - z_{\la,\mu})^{-1} \| \le \frac1{|\Im z_{\la,\mu}|} = \frac{(c + \mu)^2 + b^2}{\mu |b|}.
	\end{equation*}
	Hence
	\begin{equation*}
		\begin{aligned}
			\opH_q^{\frac12} \opT(\la)^{-1} \opH_q^{\frac12} &= (\opH_q^{\frac12} \opT(\mu)^{-1} \opH_q^{\frac12} - z_{\la,\mu})^{-1} \bigg(\frac{\mu}{\mu - \la} \opH_q^{\frac12} \opT(\mu)^{-1} \opH_q^{\frac12}\\
			&\quad + \mu \la \opH_q^{\frac12} \opT(\mu)^{-1} \opT(\la)^{-1} \opH_q^{\frac12}\bigg).
		\end{aligned}
	\end{equation*}
	Choosing $\mu = |b|$ and noticing
	\begin{equation*}
		\| (\opH_q^{\frac12} \opT(|b|)^{-1} \opH_q^{\frac12} - z_{\la,|b|})^{-1} \| \le \frac{(c + |b|)^2 + b^2}{b^2} \ls 1, \quad |b| \to +\infty,
	\end{equation*}
	we obtain (applying \eqref{eq:Hq12.Tmuminus12.norm}, \eqref{eq:Hq12.Tmuminus1.norm} and \eqref{eq:Hq12.Tlaminus1.norm}) 
	\begin{equation*}
		\begin{aligned}
			\| \opH_q^{\frac12} \opT(\la)^{-1} \opH_q^{\frac12} \| &\ls \frac{|b|}{((c + |b|)^2 + b^2)^{\frac12}} \| \opH_q^{\frac12} \opT(|b|)^{-\frac12} \| \| \opT(|b|)^{-\frac12} \opH_q^{\frac12} \|\\
			&\quad + |b| (c^2 + b^2)^{\frac12} \| \opH_q^{\frac12} \opT(|b|)^{-1} \| \| \opT(\la)^{-1} \opH_q^{\frac12} \|\\
			&\ls_K 1 + b^2 |b|^{-1} \ls_K |b|, \quad |b| \to +\infty.
		\end{aligned}
	\end{equation*}
	Returning with this estimate to \eqref{eq:G.resnorm.Hq12u1.est1}, we obtain
	\begin{equation}
		\label{eq:G.resnorm.Hq12u1.est2}
		\begin{aligned}
			\| \opH_q^{\frac12} u_1 \| &\ls_K |b|^{-1} (\| \opH_q^{\frac12} v_1 \| + |b| \| \opH_q^{\frac12} v_1 \|) + \| v_2 \|\\
			&\ls_K \| \opH_q^{\frac12} v_1 \| + \| v_2 \|, \quad |b| \to +\infty.
		\end{aligned}
	\end{equation}

	Combining \eqref{eq:G.resnorm.u2.est} and \eqref{eq:G.resnorm.Hq12u1.est2}, we have
	\begin{equation*}
		\frac{\| u \|_{\cH}}{\| v \|_{\cH}} = \frac{(\| \opH_q^{\frac12} u_1 \|^2 + \| u_2 \|^2)^{\frac12}}{\| v \|_{\cH}} \ls_K 1, \quad |b| \to +\infty.
	\end{equation*}
	Since $0 \ne v$ is arbitrary and $\CcR \times \CcR$ is dense in $\cH$, it follows that $\| (\opG - \la)^{-1} \| \ls_K 1$ as $|b| \to +\infty$, which concludes the proof.
\end{proof}

\section{An example}
\label{sec:examples}
To illustrate our results, we study the operator $\opG$ associated with the damping and potential
\begin{equation}
	\label{eq:ex.a.q.def}
	a(x) = x^2, \quad q(x) = \kappa x^2, \; \kappa > 0, \quad x \in \R. 
\end{equation}
These functions satisfy Assumption~\ref{asm:a.q.dwe} and consequently we deduce that $\opG$ is m-dissipative and, in particular, $\sigma(\opG) \subset \overline \C_{-}$. Our next result provides a description of $\sigma(\opG)$ and the behaviour of the $C_0$-semigroup of contractions generated by $\opG$.

\begin{proposition}
	\label{prop:ex.spec.G}
	Let $\opG$ denote the linear operator defined in the statement of Theorem~\ref{thm:resolvent.G} with $a(x)$ and $q(x)$ determined by \eqref{eq:ex.a.q.def}. Then the following hold.
	\begin{enumerate} [\upshape (i)]
		\item \label{itm:ex.spectrum.G} The spectrum of $\opG$ is
		\begin{equation}
			\label{eq:ex.spectrum.G}
			\sigma(\opG) = (-\infty, -\kappa/2] \sqcup \{\la_n^r, \la_n^i, \overline{\la_n^i}: n \in \N_0\}
		\end{equation}
		where, for each $n \in \N_0$, the numbers $\la_n^r, \la_n^i, \overline{\la_n^i}$ are the solutions of the equation
		\begin{equation}
			\label{eq:ex.eigenv.eq}
			\la^4 - 2 (2n + 1)^2 \la - (2n + 1)^2 \kappa = 0
		\end{equation}
		satisfying $\la \in \overline \C_{-} \setminus (-\infty, -\kappa/2]$. Moreover, as $n \to +\infty$
		\begin{align}
			\label{eq:ex.lanr.asymp}\la_n^r &= - \frac{\kappa}{2} \left(1 - \frac{2^{-\frac83}}{3} \kappa^2 (2n + 1)^{-\frac43} + o(\kappa^2 (2n + 1)^{-\frac43})\right),\\
			\label{eq:ex.lani.asymp}\la_n^i &= 2^{\frac13} (2 n + 1)^{\frac23} \left(1 - \frac{2^{-\frac73}}{3} \kappa (2 n + 1)^{-\frac23} + o(\kappa (2 n + 1)^{-\frac23})\right) e^{i (\pi - \theta_n)},
		\end{align}
		with
		\begin{equation}
			\label{eq:ex.thetan.asymp}
			\theta_n = \arctan\left(\sqrt{3} \left(1 + \frac{2^{-\frac13}}{3} \kappa (2 n + 1)^{-\frac23} + o(\kappa (2 n + 1)^{-\frac23})\right)\right).
		\end{equation}
		\item \label{itm:ex.sg.G} With definitions \eqref{eq:spec.bound.def} and \eqref{eq:growth.bound.def}, we have
		\begin{equation*}
			\omega_0 = s(\opG) < 0.
		\end{equation*}
	\end{enumerate}
\end{proposition}

\begin{remark}
	\label{rmk:ex.la.limit}
	It is clear from the above asymptotic expansions that $\la_n^r \to -\kappa/2$ as $n \to +\infty$ (i.e. the sequence of real eigenvalues of $\opG$ accumulates on $-\kappa/2$, in line with Remark~3.3 in \cite{Freitas-2018-264}). Furthermore, $|\la_{n}^i| \to +\infty$ and $\arg(\la_{n}^i) \to 2\pi/3$ as $n \to +\infty$.
\end{remark}

\begin{proof}[Proof of Proposition~\ref{prop:ex.spec.G}]
	\ref{itm:ex.spectrum.G} Since $a(x) = \kappa^{-1} q(x)$ for every $x \in \R$, applying \cite[Thm.~3.2]{Freitas-2018-264} and \cite[Rmk.~3.3]{Freitas-2018-264}, we deduce that $\sigma_{e2}(\opG) \subset (-\infty, -\kappa/2]$. Moreover, $\sigma_{e2}(\opG)$ is closed and it is therefore sufficient to show $(-\infty, -\kappa/2) \subset \sigma_{e2}(\opG)$ to conclude $\sigma_{e2}(\opG) = (-\infty, -\kappa/2]$. In order to do this, selecting an arbitrary $\la \in (-\infty,-\kappa/2)$, we will construct a singular sequence $(\Phi_{n})_{n \in \N} \subset \Dom(\opG)$ for $\la$ adapting the proof of \cite[Thm.~4.2]{Freitas-2018-264}.
	
	Letting
	\begin{equation*}
		A(x) := -(q(x) + 2 \la a(x) + \la^2) = 2 |\la + \kappa/2| x^2 - \la^2, \quad x \in \R,
	\end{equation*}
	we note that $A \in C^1(\R)$ and
	\begin{equation}
		\label{eq:ex.A.infty}
		\lim_{x \to +\infty} A(x) = +\infty, \quad \lim_{x \to +\infty} \frac{|A'(x)|}{A(x)} = 0.
	\end{equation}
	Our main goal is to find a sequence $(\phi_n)_{n \in \N} \subset \WttR \cap \Dom(x^2)$ such that (the infimum of) $\supp \phi_n$ goes to infinity in $\Rplus$ as $n \to +\infty$ and
	\begin{equation}
		\label{eq:ex.sing.seq.conv}
		\lim_{n \to +\infty} \frac{\| \Dtp \phi_n + A \phi_n \|}{\| \Ntp \phi_n \|} = 0.
	\end{equation}
	By defining $\Phi_{n} := (\phi_n, \la \phi_n)^t$, it follows from each of the above two properties of $(\phi_n)_{n \in \N}$ that $\Phi_{n}/\|\Phi_{n}\|_{\cH} \overset{w}{\to} 0$ as $n \to +\infty$ and
	\begin{equation*}
		\frac{\| (\opG - \la) \Phi_{n} \|_{\cH}}{\| \Phi_{n} \|_{\cH}} \le \frac{\| \Dtp \phi_n + A \phi_n \|}{\| \Ntp \phi_n \|} \to 0, \quad n \to +\infty,
	\end{equation*}
	respectively.
	
	Letting $\alpha_{\la} := |\la|/\sqrt{2 |\la + \kappa/2|}$ and applying \eqref{eq:ex.A.infty}, we have
	\begin{gather*}
		A(x) > 0, \quad x > \alpha_{\la},\\
		\rho_n := \underset{t > n}{\sup} \frac{|A'(t)|}{A(t)} \to 0, \quad n > \alpha_{\la}, \quad n \to +\infty.
	\end{gather*}
	Let us define $\phi_n(x) := \varphi_n(x) \psi_{\la}(x)$, $x \in \R$, $n \in \N$, where
	\begin{align*}
		\psi_{\la}(x) &:= \exp\left(i \int_{\alpha_{\la}}^{x} (A(t))^{\frac12} \dd t\right), \quad x \ge \alpha_{\la},\\
		\varphi_n(x) &:= \rho_n^{\frac14} \varphi(\rho_n^{\frac12} x - n), \quad \varphi \in C_c^{\infty}((0,1)), \quad \| \varphi \| = 1.
	\end{align*}
	From these definitions, we immediately deduce: (i) $\rho_n^{-\frac12} n \to +\infty$, as $n \to +\infty$, and (ii) $\supp \varphi_n \subset (\rho_n^{-\frac12} n, \rho_n^{-\frac12} (n + 1))$, $n \in \N$. Furthermore
	\begin{align*}
		\| \varphi_n \|^2 &= \rho_n^{\frac12} \intR |\varphi(\rho_n^{\frac12} x - n)|^2 \dd x = 1, \quad n \in \N,\\
		\| \varphi'_n \|^2 &= \rho_n^{\frac32} \intR |\varphi'(\rho_n^{\frac12} x - n)|^2 \dd x = \| \varphi' \|^2 \rho_n = o(1), \quad n \to +\infty,\\
		\| \varphi''_n \|^2 &= \rho_n^{\frac52} \intR |\varphi''(\rho_n^{\frac12} x - n)|^2 \dd x = \| \varphi'' \|^2 \rho_n^2 = o(1), \quad n \to +\infty,
	\end{align*}
	and for $x \ge \alpha_{\la}$
	\begin{align*}
		\psi'_{\la}(x) &:= i (A(x))^{\frac12} \psi_{\la}(x),\\
		\psi''_{\la}(x) &:= i \frac12 (A(x))^{-\frac12} A'(x) \psi_{\la}(x) - A(x) \psi_{\la}(x).
	\end{align*}
	Straightforward calculations show that as $n \to +\infty$
	\begin{equation*}
		\| \Ntp \phi_n \| \ge \| \varphi_n \psi'_{\la} \| - \| \varphi'_n \| \gs \rho_n^{-\frac12} n \implies \| \Ntp \phi_n \|^{-1} = \BigO(\rho_n^{\frac12} n^{-1}),
	\end{equation*}
	and
	\begin{equation*}
		\begin{aligned}
			&\| \Dtp \phi_n + A \phi_n \| \le \| \varphi''_n \| + 2 \| A^{\frac12} \varphi'_n \| + \frac12 \| A^{-\frac12} A' \varphi_n \| \ls \rho_n + n + 1 + 1\\
			&\implies \| \Dtp \phi_n + A \phi_n \| = \BigO(n).
		\end{aligned} 
	\end{equation*}
	Therefore $\| \Dtp \phi_n + A \phi_n \| / \| \Ntp \phi_n \| = \BigO(\rho_n^{\frac12})$ which shows that \eqref{eq:ex.sing.seq.conv} holds. We conclude that $(\Phi_{n})_{n \in \N}$ is indeed a singular sequence and $\la \in \sigma_{e2}(\opG)$ as claimed.
	
	To determine the eigenvalues of $\opG$, we apply the spectral equivalence \eqref{eq:G.Tla.spectral.equivalence} (see also \cite[Thm.~3.2]{Freitas-2018-264} and \cite[Rmk.~3.3]{Freitas-2018-264}) and seek to find the set of $\la \in \overline \C_{-} \setminus (-\infty, -\kappa/2]$ such that $0 \in \sigma_p(\opT(\la))$, where
	\begin{equation*}
		\opT(\la) = \Dt + (\kappa + 2\la) x^2 + \la^2, \quad \Dom(\opT(\la)) = W^{2,2}(\R) \cap \Dom(x^2),
	\end{equation*}
	i.e. we need to find every $\la \in \overline \C_{-} \setminus (-\infty, -\kappa/2]$ such that
	\begin{equation}
		\label{eq:ex.eigenvalue.problem}
		-u_{\la}''(x) + (2 \la + \kappa) x^2 u_{\la}(x) = -\la^2 u_{\la}(x), \quad x \in \R,
	\end{equation}
	for some $0 \ne u_{\la} \in \Dom(\opT(\la))$.
	
	To this end, with $\gamma \in \C$ and $|\arg(\gamma)| < \pi$, let us consider the $\gamma$-dependent Schr\"odinger operator family
	\begin{equation*}
		\opH_{\gamma} := \Dt + \gamma x^2, \quad \Dom(\opH_{\gamma}) := W^{2,2}(\R) \cap \Dom(x^2).
	\end{equation*}
	It has been shown that $\opH_{\gamma}$ is a family of closed operators with compact resolvent (see \cite[Lem.~2.3]{Freitas-2018-264}). Furthermore, the spectrum of the rotated operator $\widetilde \opH_{\gamma} := -\gamma^{-\frac12} \Dtp + \gamma^{\frac12} x^2$ is independent of $\gamma$ (see \cite[Lem.~5]{Davies-2000-32}). Since $\sigma(\widetilde \opH_{1}) = \{2 n + 1 : n \in \N_0\}$, with corresponding eigenfunctions $\tilde{u}_n(x) = H_n(x) \exp(-x^2/2)$, $x \in \R$, $n \in \N_0$, where $H_n$ are the Hermite polynomials (see e.g. \cite[Sec.~1.3]{Helffer-2013-book}), it follows that $\sigma(\opH_{\gamma}) = \{(2 n + 1) \gamma^{\frac12} : n \in \N_0\}$ with eigenfunctions $u_n(x) = \tilde{u}_n(\gamma^{\frac14} x)$, $x \in \R$, $n \in \N_0$. To verify that $(u_n)_{n \in \N_0} \subset \Dom(\opH_{\gamma})$, consider $u_0(x) = \exp(-\gamma^{\frac12} x^2/2)$ and observe that $|\arg(\gamma)| < \pi$ and $x \in \R$ together imply that $\Re (\gamma^{\frac12}) x^2 > 0$ ($x \ne 0$). Hence we conclude that $u_0$, its derivatives and its product with any polynomial belong to $\Lt(\R)$.
	
	Therefore, setting $\gamma = 2 \la + \kappa$, the solutions of the eigenvalue problem \eqref{eq:ex.eigenvalue.problem} must satisfy the family of equations
	\begin{equation}
		\label{eq:ex.quartic}
		\la^4 = (2 n + 1)^2 (2 \la + \kappa), \quad n \in \N_0.
	\end{equation}
	Each of the above equations is a (reduced) quartic which can be solved in a standard way by Ferrari's method re-casting it as a product of two quadratics. To do this, we re-write \eqref{eq:ex.quartic} as
	\begin{equation*}
		(\la^2 + y)^2 = \left(\sqrt{2 y} \la + \frac{(2 n + 1)^2}{\sqrt{2 y}}\right)^2,
	\end{equation*}
	where $y \ne 0$ is a solution of
	\begin{equation*}
		y^3 + (2 n + 1)^2 \kappa y - \frac12 (2 n + 1)^4 = 0.
	\end{equation*}
	One can verify that this cubic has a (real) root given by Cardano's formula (with $\kappa_n := (16/27) \kappa^3 (2 n + 1)^{-2} > 0$)
	\begin{equation}
		\label{eq:yn.def}
		y_n = 2^{-\frac23} (2 n + 1)^{\frac43} \left(\left(\left(1 + \kappa_n\right)^{\frac12} + 1\right)^{\frac13} - \left(\left(1 + \kappa_n\right)^{\frac12} - 1\right)^{\frac13}\right), \quad n \in \N_0.
	\end{equation}
	Noting that $0 < y_n < 2^{-\frac13} (2 n + 1)^{\frac43}$, for every $n \in \N_0$, the solutions of \eqref{eq:ex.quartic} are (with $y_{n,\pm} := (4 (2 n + 1)^2 (2 y_n)^{-\frac32} \pm 1)^{\frac12} > 0$)
	\begin{equation*}
		\begin{aligned}
			\la_{n,\pm}^r &:= -\frac12 (2 y_n)^{\frac12} (-1 \mp y_{n,-}),\\
			\la_{n, \pm}^i &:= -\frac12 (2 y_n)^{\frac12} (1 \mp i y_{n,+}),
		\end{aligned}
	\end{equation*}
	for every $n \in \N_0$. Next we examine each of these two sets of roots in turn.
	
	Since $\opG$ is m-dissipative, any eigenvalues must lie in the semi-plane $\{\la \in \C : \Re\la \le 0\}$. Hence, appealing once more to the spectral equivalence \eqref{eq:G.Tla.spectral.equivalence}, we discard $\la_{n,+}^r$ for every $n \in \N_0$ as admissible solutions of the problem \eqref{eq:ex.eigenvalue.problem}. Denoting $\la_n^r := \la_{n,-}^r$, straightforward but somewhat lengthy calculations show that $\la_{n}^r < 0$ for every $n \in \N_0$ and we have asymptotically
	\begin{equation*}
		\la_{n}^r = -\frac{\kappa}{2} \left(1 - \frac{2^{-\frac83}}{3} \kappa^2 (2 n + 1)^{-\frac43} + o(\kappa^2 (2 n + 1)^{-\frac43})\right), \quad n \to +\infty,
	\end{equation*}
	as claimed in \eqref{eq:ex.lanr.asymp}. Let us now consider whether any $\la_{n} \in (-\infty,-\kappa/2]$ can be a solution of the eigenvalue problem. If $\la_{n} + \kappa/2 = 0$, equation \eqref{eq:ex.eigenvalue.problem} becomes
	\begin{equation*}
		-u''(x) + \la_{n}^2 u(x) = 0,
	\end{equation*}
	whose general solution $u_n(x) = C_1 \exp(|\la_{n}| x) + C_2 \exp(-|\la_{n}| x)$ does not belong to $\Lt(\R)$ unless $C_1 = C_2 = 0$. For $\la_{n} + \kappa/2 < 0$, applying the change of variable $y = \sqrt[4]{4 |2\la_{n} + \kappa|} x$ enables us to re-write \eqref{eq:ex.eigenvalue.problem} as
	\begin{equation}
		\label{eq:ex.Weber.eq}
		v''(y) + \left(\frac14 y^2 - b_n\right) v(y) = 0,
	\end{equation}
	with $b_n := \la_{n}^2 / \sqrt{4 |2 \la_{n} + \kappa|} > 0$. Note that equation \eqref{eq:ex.Weber.eq} does not have $\Lt$ solutions. This is a standard result from the theory of Sturm-Liouville operators in the positive half-line case (see e.g. \cite[Thm.~5.10]{Titchmarsh-1962-book1} or \cite[Thm.~3.5.6]{eastham1982schrodinger}). For the whole line, the problem can be reduced to the positive half-line with Dirichlet (or Neumann) boundary condition at $0$. From this analysis, we conclude that only $\la_{n} \in (-\kappa/2, 0)$ are admissible real solutions of \eqref{eq:ex.eigenvalue.problem}. Hence (noting also Remark~\ref{rmk:ex.la.limit}) we find
	\begin{equation}
		{\label{eq:ex.spec.bound.re}}
		-\frac{\kappa}{2} < \sup\{\la_{n}^r \in \sigma_{p}(\opG) : n \in \N_0\} < 0.
	\end{equation}

	On the other hand, every imaginary solution $\la_{n, \pm}^i$, $n \in \N_0$, is admissible because $\Re \la_{n, \pm}^i = -(1/2) (2 y_n)^{\frac12} < 0$ and $\Im \la_{n, \pm}^i = \pm (1/2) (2 y_n)^{\frac12} y_{n,+} \ne 0$. Moreover, it is straightforward to verify that $y_n$ is an increasing function of $n$ and hence (with $\la_n^i := \la_{n,+}^i$)
	\begin{equation}
		{\label{eq:ex.spec.bound.im}}
		-2^{-\frac23} < \sup\{\Re \la_{n}^i : n \in \N_0\} = -\frac12 (2 y_0)^{\frac12} < 0.
	\end{equation}
	Furthermore we find that for every $n \in \N_0$
	\begin{align*}
		|\la_{n}^i| &= (2 n + 1) (2 y_n)^{-\frac14} \left(1 + \frac12 (2 n + 1)^{-2} (2 y_n)^{\frac32}\right)^{\frac12},\\
		\arg(\la_{n}^i) &= \pi - \arctan\left(\frac{2 (2 n + 1) y_{n,+}}{(2 y_n)^{\frac34}}\right).
	\end{align*}
	Asymptotic expansions as $n \to +\infty$ yield
	\begin{equation*}
		\begin{aligned}
			\frac{2 n + 1}{(2 y_n)^{\frac14}} \left(1 + \frac12 \frac{(2 y_n)^{\frac32}}{(2 n + 1)^{2}}\right)^{\frac12} &= 2^{\frac13} (2 n + 1)^{\frac23} \Bigg(1 - \frac{2^{-\frac73}}{3} \kappa (2 n + 1)^{-\frac23}\\
			&\hspace{1in} + o(\kappa (2 n + 1)^{-\frac23})\Bigg),\\
			\frac{2 (2 n + 1) y_{n,+}}{(2 y_n)^{\frac34}} &= \sqrt{3} \left(1 + \frac{2^{-\frac13}}{3} \kappa (2 n + 1)^{-\frac23} + o(\kappa (2 n + 1)^{-\frac23})\right),
		\end{aligned}
	\end{equation*}
	which proves \eqref{eq:ex.lani.asymp} and \eqref{eq:ex.thetan.asymp}.
	
	\ref{itm:ex.sg.G} The claim follows immediately from Corollary~\ref{cor:sg.G.decay}.
\end{proof}

To illustrate this result, we plot the spectrum of $\opG$ corresponding to $q(x) = 10 x^2$ and $a(x) = x^2$ (see Fig.~\ref{fig:ex.spec.G.graph}). In this case, $\sigma_{e2}(\opG) = (-\infty, -5]$ and we calculate
$\sup\{\la_{n}^r : n \in \N_0\} = -1.61326$ and $\sup\{\Re\la_{n}^i : n \in \N_0\} = -0.15809$. This yields the decay rate $\omega_0(\opG) = -0.15809$ for the associated semigroup.
\begin{figure}[h]
	\includegraphics[scale=0.95]{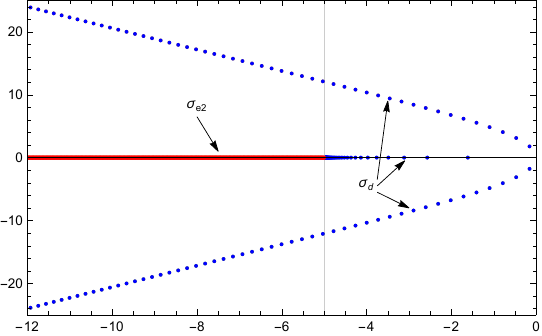}
	\caption{The spectrum of $\opG$ for $q(x) = 10 x^2$ and $a(x) = x^2$.}
	\label{fig:ex.spec.G.graph}
\end{figure}

\subsection{Further remarks}
\label{ssec:example.remarks}
From \eqref{eq:ex.spec.bound.re}, \eqref{eq:spec.bound.def} and \eqref{eq:omega0.eq.sG}, it is clear that
\begin{equation*}
	\kappa \to 0^{+} \implies \omega_0 \to 0^{-}.
\end{equation*}
It can also be verified from \eqref{eq:yn.def} that $y_0 \to 0$ as $\kappa \to +\infty$ and therefore by \eqref{eq:ex.spec.bound.im}, \eqref{eq:spec.bound.def} and \eqref{eq:omega0.eq.sG} we find
\begin{equation*}
	\kappa \to +\infty \implies \omega_0 \to 0^{-}.
\end{equation*}
Hence the exponential decay rates become progressively weaker in both limiting cases although for different reasons from a spectral point of view: when $\kappa$ approaches zero, it is the (real) essential spectrum that approaches the imaginary axis whereas for very large $\kappa$ the phenomenon is driven by the (non-real) eigenvalues.
\begin{figure}[h]
	\includegraphics[scale=0.55]{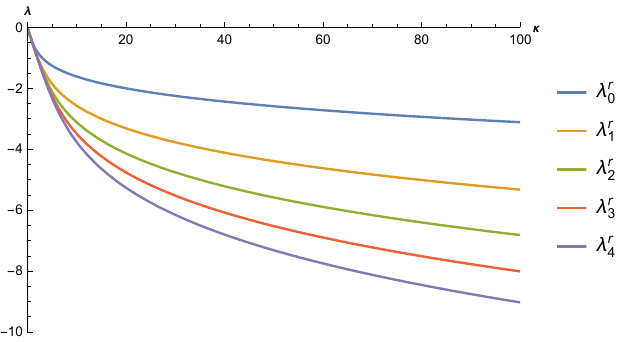}\\[\smallskipamount]
	\includegraphics[scale=0.55]{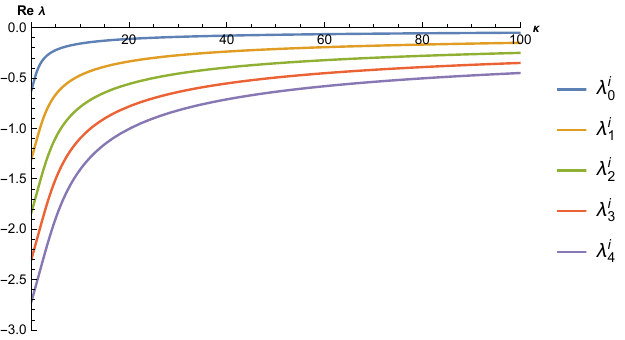}\hfill
	\includegraphics[scale=0.55]{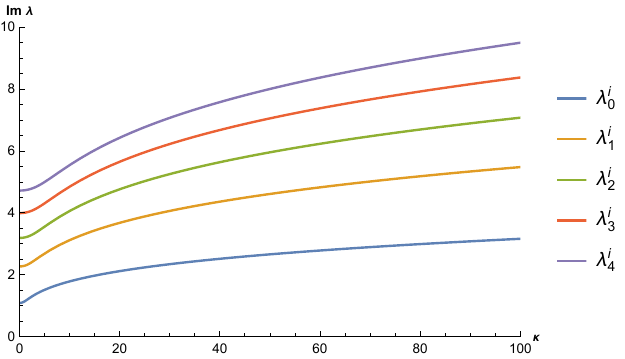}
	\caption{First 5 real (top) and first 5 imaginary (bottom) eigenvalues of $\opG$ as a function of $\kappa$.}
	\label{fig:ex.5.eig.GV}
\end{figure}

We also note that, when we remove the potential $q$ altogether (i.e. when $\kappa = 0$), we recover the spectral configuration shown in Figure~1 in \cite{Freitas-2018-264}. In terms of the behaviour of the solutions of the Cauchy problem, this of course implies that their time-decay is no longer exponential.

We illustrate the dependency on $\kappa$ of the first 5 real and the first 5 (upper-half plane) imaginary eigenvalues of $\opG$ in Fig.~\ref{fig:ex.5.eig.GV}. It highlights the fact that $\sigma(\opG)$ will touch the imaginary axis when $q(x) = 0$ as the set of real eigenvalues vanishes and the essential spectrum fills $\overline \Rminus$. Note that the non-real point spectrum is not empty and it stays away from the imaginary axis (although as noted above it will get closer to it as $\kappa$ becomes larger).

We conclude with the observation (without making specific claims) that the above analysis can also be applied to  more general even monomials $a(x) = x^{2 m}$, $q(x) = \kappa x^{2 m}$, $\kappa > 0$, $m \in \N$, similarly to the study in Proposition~6.1 in \cite{Freitas-2018-264}, with the eigenvalues of the corresponding self-adjoint anharmonic oscillator replacing $\sigma(\widetilde \opH_{1})$ above.

\bibliography{references}
\bibliographystyle{acm}	

\end{document}

%% file: commands.tex
\usepackage{amsmath,amsthm, amssymb}
\usepackage{mathrsfs}
\usepackage[shortlabels]{enumitem}
\usepackage{graphicx}
\usepackage[font=small]{caption}
\usepackage{xcolor}
\usepackage{url}

\newcommand{\N}{\mathbb{N}}
\newcommand{\R}{{\mathbb{R}}}
\newcommand{\C}{{\mathbb{C}}}

\newcommand{\dd}{{{\rm d}}}

\newcommand{\odd}{\overline{\dd}}

\newcommand{\ie}{{\emph{i.e.}}}
\newcommand{\eg}{{\emph{e.g.}}}

\newcommand{\la}{\lambda}

\newcommand{\eps}{\varepsilon}

\newcommand{\essinf}{\operatorname*{ess \,inf}}

\newcommand{\Dom}{{\operatorname{Dom}}}
\newcommand{\Ker}{{\operatorname{Ker}}}

\renewcommand{\Re}{\operatorname{Re}}
\renewcommand{\Im}{\operatorname{Im}}
\newcommand{\dist}{\operatorname{dist}}

\newcommand{\supp}{\operatorname{supp}}

\newcommand{\Num}{\operatorname{Num}}

\newcommand{\loc}{\mathrm{loc}}
\newcommand{\BigO}{\mathcal{O}}

\newcommand{\rad}{{\operatorname{rad}}}
\newcommand{\opH}{H}

\newcommand{\opA}{A}
\newcommand{\opB}{B}

\newcommand{\opG}{G}
\newcommand{\opI}{I}

\newcommand{\opK}{K}
\newcommand{\opP}{P}
\newcommand{\opQ}{Q}
\newcommand{\opR}{R}
\newcommand{\opS}{S}
\newcommand{\opT}{T}
\newcommand{\opU}{U}

\newcommand{\opOP}{\operatorname{OP}}
\newcommand{\Core}{{\mathcal{C}}}
\newcommand{\Dcore}{{\mathcal{D}}}

\newcommand{\Rplus}{\R_+}
\newcommand{\Rminus}{\R_-}

\newcommand{\CiR}{{C^{\infty}(\R)}}
\newcommand{\SchwR}{{\sS(\R)}}

\newcommand{\intR}{\int_{\R}}

\newcommand{\Rd}{\mathbb{R}^d}

\newcommand{\Lt}{{L^2}}

\newcommand{\LiR}{{L^{\infty}(\R)}}

\newcommand{\LiOm}{{L^{\infty}(\Omega)}}

\newcommand{\LilocR}{{L^{\infty}_{\rm loc}(\R})}

\newcommand{\CcR}{{C_c^{\infty}(\R)}}
\newcommand{\CcOm}{{C_c^{\infty}(\Omega)}}

\newcommand{\WotR}{{W^{1,2}(\R)}}

\newcommand{\WttR}{{W^{2,2}(\R)}}

\newcommand{\Dt}{-\partial_x^2}

\newcommand{\Dtp}{\partial_x^2}
\newcommand{\Nt}{-\partial_x}

\newcommand{\Ntp}{\partial_x}
\newcommand{\Dtime}{\partial_t^2}
\newcommand{\Ntime}{\partial_t}

\newcommand{\ls}{\lesssim 	}
\newcommand{\gs}{\gtrsim}

\theoremstyle{plain}

\newtheorem{theorem}{Theorem}[section]
\newtheorem{lemma}[theorem]{Lemma}

\newtheorem{proposition}[theorem]{Proposition}
\newtheorem{corollary}[theorem]{Corollary}

\theoremstyle{definition}
\newtheorem{example}[theorem]{Example}
\newtheorem{remark}[theorem]{Remark}
\newtheorem{asm-sec}[theorem]{Assumption}

\newcommand\cB{\mathcal B}
\newcommand\cD{\mathcal D}

\newcommand\cH{\mathcal H}

\newcommand\cS{\mathcal S}

\newcommand\cW{\mathcal W}

\newcommand\sF{\mathscr F}
\newcommand\sL{\mathscr L}
\newcommand\sS{\mathscr S}